\documentclass[AMA,STIX1COL]{WileyNJD-v2_modified}

\articletype{Research Article}

\received{xxx}
\revised{yyy}
\accepted{zzz}
  
\usepackage[percent]{overpic}
\usepackage{tikz}

\DeclareMathOperator{\spt}{spt}

\newtheorem{example}{Example}

\newcommand{\R}{\mathbb{R}}

\newcommand{\rhomin}{\rho_{\textup{min}}}
\newcommand{\yii}{y_{i+1}}
\newcommand{\yi}{y_i}
\newcommand{\wmax}{v_{\textup{max}}}
\newcommand{\vmax}{v_{\textup{max}}}
\renewcommand{\next}{\textsc{next}}
\newcommand{\yip}{y_{i,\alpha}}

\newcommand{\DG}{\mathcal D_{\!\mathcal N}}
\newcommand{\Dftl}{\mathcal D^{FtL}_p}
\newcommand{\Dftlnop}{\mathcal D^{FtL}}
\newcommand{\Dlwr}{\mathcal D^{LWR}_p}
\newcommand{\source}{\flat}
\newcommand{\dest}{\sharp}
\newcommand{\muuno}{\mu^\source}
\newcommand{\mudue}{\mu^\dest}
\newcommand{\nuuno}{\nu^\source}
\newcommand{\nudue}{\nu^\dest}

\newcommand{\ybfuno}{\mathbf y^\source}
\newcommand{\ybfdue}{\mathbf y^\dest}
\newcommand{\yuno}{y^\source}
\newcommand{\ydue}{y^\dest}
\newcommand{\ybfbaruno}{\bar{\mathbf y}^\source}
\newcommand{\ybfbardue}{\bar{\mathbf y}^\dest}
\newcommand{\rhouno}{\rho^\source}
\newcommand{\rhodue}{\rho^\dest}
\newcommand{\rhobaruno}{\bar\rho^\source}
\newcommand{\rhobardue}{\bar\rho^\dest}
\newcommand{\Mtot}{\mathcal M}

\newcommand{\vmaxuno}{\vmax^{\source}}
\newcommand{\vmaxdue}{\vmax^{\dest}}
\newcommand{\vuno}{v^{\source}}
\newcommand{\vdue}{v^{\dest}}
\newcommand{\wuno}{w^{\source}}
\newcommand{\wdue}{w^{\dest}}
\newcommand{\Funo}{F^{\source}}
\newcommand{\Fdue}{F^{\dest}}

\begin{document}

\title{
Comparing comparisons between vehicular traffic states in microscopic and macroscopic first-order models
}

\author[1]{Emiliano Cristiani*}

\author[2]{Maria Cristina Saladino}

\authormark{CRISTIANI \textsc{et al}}

\address[1]{\orgdiv{Istituto per le Applicazioni del Calcolo}, \orgname{Consiglio Nazionale delle Ricerche}, \orgaddress{\state{Rome}, \country{Italy}}}

\address[2]{\orgdiv{Dipartimento di Matematica}, \orgname{Sapienza, Universit\`a di Roma}, \orgaddress{\state{Rome}, \country{Italy}}}

\corres{*E. Cristiani, IAC--CNR, Via dei Taurini, 19 -- 00185, Rome, Italy. \email{e.cristiani@iac.cnr.it}}


\abstract[Summary]{
In this paper we deal with the analysis of the solutions of traffic flow models at multiple scales, both in the case of a single road and of road networks. We are especially interested in measuring the distance between traffic states (as they result from the mathematical modeling) and investigating whether these distances are somehow preserved passing from the microscopic to the macroscopic scale. By means of both theoretical and numerical investigations, we show that, on a single road, the notion of Wasserstein distance fully catches the human perception of distance independently of the scale, while in the case of networks it partially loses its nice properties.
}


\maketitle

\footnotetext{\textbf{Abbreviations:} 
ODE, ordinary differential equation;
PDE, partial differential equation;
LWR, Ligthill-Whitham-Richards; 
ARZ, Aw-Rascle-Zhang; 
FtL, Follow-the-Leader;
LP, linear programming
}

\graphicspath{{./}{figures/}}

\section{Introduction}\label{sec:intro}
In this paper we deal with the analysis of the solutions of traffic flow models at multiple scales. 
More precisely, we are interested in measuring the distance between traffic states (as they result from the mathematical modeling) and investigating whether these distances are somehow preserved passing from the microscopic to the macroscopic scale. We will consider both the case of a single road and of road networks.

Connections between microscopic (agent-based) and macroscopic (fluid-dynamics) traffic flow models are already well established. 
Aw et al.\cite{aw2002SIAP}, Greenberg\cite{greenberg2001SIAP}, and Di Francesco et al.\cite{difrancesco2017MBE} investigated the many-particle limit in the framework of second-order traffic models, deriving the macroscopic ARZ\cite{aw2000SIAP,zhang2002TRB} model from a particular second-order microscopic FtL\cite{helbing2001RMP,pipes1953JAP} model.
Instead Colombo and Rossi,\cite{colombo2014RSMUP} Rossi,\cite{rossi2014DCDS-S} Di Francesco and Rosini,\cite{difrancesco2015ARMA} and Di Francesco et al.\cite{difrancesco2017BUMI} investigated the many-particle limit in the framework of first-order traffic models, deriving the LWR\cite{lighthill1955PRSLA,richards1956OR} model as the limit of a specific first-order FtL model.
Let us also mention the papers by Forcadel et al.,\cite{forcadel2017DCDS-A} Forcadel and Salazar\cite{forcadel2015DIE} which investigate the many-particle limit exploiting the link between conservation laws and Hamilton-Jacobi equations.

Moving to road networks, analogous connections are rarer. This is probably due to the fact that macroscopic traffic models on networks are in general ill-posed, since the conservation of the mass is not sufficient alone to characterize a unique solution at junctions. This ambiguity makes more difficult to find the right limit of the microscopic model, which, in turn, can be defined in
different ways near the junctions. In this context let us mention the paper by Cristiani and Sahu\cite{cristiani2016NHM} which investigates the many-particle limit of a first-order FtL model suitably extended to a road network. The corresponding macroscopic model appears to be the extension of the LWR model on networks introduced by Hilliges and Weidlich\cite{hilliges1995TRB} and then extensively studied by Briani and Cristiani\cite{briani2014NHM} and Bretti et al.\cite{bretti2014DCDS-S}

This paper focuses in particular on the comparison of solutions to the equations associated to first-order traffic models. It is useful to note here that many problems require the comparison of traffic states, or, in other words, the computation of the ``distance'' between two vehicle arrangements (at microscopic scale) or vehicle density distributions (at macroscopic scale). For example,
\begin{itemize}
\item \textit{Theoretical study} of scalar conservation laws.
\item \textit{Sensitivity analysis} (quantifying how the uncertainty in the outputs of a model can be apportioned to different sources of uncertainty in the inputs).
\item \textit{Convergence} of the numerical schemes (verifying that the numerical solution is close to the exact solution).
\item \textit{Calibration} (finding the values of the parameters for the predicted outputs to be as close as possible to the observed ones).
\item \textit{Validation} (checking if the outputs are close to the observed ones).
\end{itemize}
The question arises as to which notion of distance is the more appropriate in our context. 
Regarding conservation laws, which are the foundations of macroscopic models like LWR and ARZ, it is nowadays quite established that $L^p$-distances are inadequate for comparing solutions, i.e.\ vehicular density functions, since they lead to results which can mismatch the human perception of physical distance.\cite[Sect.\ 7.1]{cristianibook}$^,$\cite{briani2018CMS} 
For example, the $L^p$-distance between two density functions $\rhouno$ and $\rhodue$ with equal total mass $\int_\R \rhouno dx=\int_\R \rhodue dx=1$ and disjoint supports is equal to 2, \emph{regardless how ``far'' the masses are}.
A notion of distance which instead better matches the intuition is that of Wasserstein distance (also known as $Lip^\prime$-norm, earth mover's distance, $\bar d$-metric, Mallows distance), which is able to quantify the effort needed to transport a certain amount of mass for a certain distance.
Unfortunately, a mere application of the Wasserstein distance in a microscopic framework would only be meaningful under the assumptions that vehicles are indistinguishable and interchangeable, which of course do not hold true (ask drivers). 

\medskip

In this paper we study some distances between microscopic and macroscopic traffic states, which are actually usable in the context of traffic flow. Then we investigate under which assumptions they coincide and we study their behavior as the number of vehicles goes to infinity, in order to catch their scale-preserving properties. 

\medskip

The paper is organized as follows. 
In Section \ref{sec:ingredients} we present the main ingredients of our investigations, i.e.\ the microscopic FtL model, the macroscopic LWR model, and the limit relation between the two models (on a single road and on road networks). We also recall the definition of Wasserstein distance. For all of these notions, references for numerical approximation are given.
In Section \ref{sec:compare} we introduce the definition of some distance functions particularly suitable for comparing traffic states at micro- and macro-scale.
In Section \ref{sec:singleroadanalysis} we investigate from the theoretical and numerical point of view the relations between the previously introduced distance functions in the case of a single road.
In Section \ref{sec:networkanalysis} we replicate the analysis of Section \ref{sec:singleroadanalysis} on a road network.
Finally, in Section \ref{sec:conclusions} we sketch some conclusions.

\section{Ingredients}\label{sec:ingredients}

\subsection{Many-particle limit on a single road}\label{sec:limit_singleroad}
Let us assume that vehicles move on a single infinite road with a single lane, and that they cannot overtake each other.

\subsubsection{FtL model}\label{sec:FtLsingleroad}
Adopting a \emph{microscopic} point of view, we assume that we are able to track the \emph{position of every single vehicle} on the road. Moreover, we assume that the dynamics of each vehicle depend only on the vehicle itself and the vehicle in front of it, so that, in a cascade, the whole traffic flow is determined by the dynamics of the very first vehicle, termed the \emph{leader}. 

We denote by $n\in\mathbb N$ the number of vehicles and by $\ell_n>0$ their length (equal for all vehicles). 
We also denote by $y_i(t)$ the position of (the barycenter of) the $i$-th vehicle and we assume that at the initial time $t=0$ vehicles are labeled in order, i.e.\ $y_1(0)<y_2(0)<\ldots <y_n(0)$. This guarantees that the $(i+1)$-th car is just in front of the $i$-th one. Moreover, it is assumed that at $t=0$ cars do not overlap, i.e.\ $\yii(0) - \yi(0)\geq\ell_n$, $i=1,\ldots,n-1$. We are now ready to introduce the model, described by the following system of $n$ ODEs
\begin{equation}\label{FTL_1road}
\left\{
\begin{array}{ll}
\dot y_i(t)=w\big(\yii(t)-\yi(t)\big), & \qquad i=1,\ldots,n-1 \\
\dot y_n(t)=\wmax,
\end{array}
\right.
\end{equation}
where $\wmax>0$ is the maximal admissible velocity and $w$, which represents the \emph{velocity} of the vehicles, is such that 
\begin{equation}\label{def:w}
w:[\ell_n,+\infty)\to[0,\wmax].
\end{equation}
Note that the $n$-th vehicle is the leader and it moves at maximal velocity. 
In the following we will denote by $\mathbf y=(y_1,\ldots,y_n)$ the solution to \eqref{FTL_1road}, duly complemented with initial conditions.

\paragraph{Numerical approximation}
System \eqref{FTL_1road} will be numerically integrated by means of the explicit Euler scheme.

\subsubsection{LWR model}\label{sec:LWRsingleroad}
Adopting a \emph{macroscopic} point of view, we assume that we are able to track only \emph{averaged quantities} like the \emph{velocity} $v$ and the \emph{density} $\rho$ (number of vehicles per length unit). 
Hereafter, we will assume without loss of generality that the density is normalized, i.e.\ $\rho=0$ corresponds to the minimal density (empty road) and $\rho=1$ corresponds to the maximal density (fully congested road).

The well known LWR\cite{lighthill1955PRSLA,richards1956OR} model describes the evolution of the density $\rho$ of vehicles by means of the following conservation law
\begin{equation}\label{LWR_1road}
\partial_t\rho + \partial_x(\rho v(\rho))=0, \qquad  t>0, \quad x\in\R.
\end{equation}
Here the unknown function $\rho(t,x)$ represents the density at time $t>0$ and point $x\in\R$, while $v$ gives the velocity as a function of the density, and must be specified to complete the model. 
We assume that  $v\in C^1([0,1];[0,\wmax])$ is any function such that  
\begin{equation}\label{properties_of_v}
v' < 0, \qquad v(0)=\wmax, \qquad v(1)=0.
\end{equation}
A typical simple choice often used for theoretical investigations, and that we will employ in the rest of the paper, is 
\begin{equation}\label{v(rho)}
v(\rho)=\wmax(1-\rho).
\end{equation}

\paragraph{Numerical approximation}
The conservation law \eqref{LWR_1road} will be numerically solved by means of the classical Godunov scheme for the LWR model, as described, e.g., in the paper by Bretti et al.\cite[Sect.\ 3.1]{bretti2007ACME}

\subsubsection{Relation between scales}\label{sec:relationbetweenscales_singleroad}
In order to study the many-particle limit of the FtL model described above, the number $n$ of vehicles must be no longer considered as a \textit{model parameter}, rather it has to be seen as a \textit{scale parameter} which goes to $+\infty$. From the physical point of view, the quantity which is actually preserved is the total length (or mass) $\Mtot>0$ of vehicles. We translate this fact by setting
\begin{equation}\label{def:elln}
\ell_n=\frac{\Mtot}{n-1}.
\end{equation} 
This relation is crucial for the micro-to-macro limit since it translates the fact that the total length does not change when the number of vehicles tends to infinity, because vehicles ``shrink'' accordingly.

In order to recall precisely the results about the correspondence between the two models, we need to introduce first the natural spaces for the macroscopic density $\rho$ and for the vectors of vehicles' positions $\mathbf y=(y_1,\ldots,y_n)$, at any fixed time. We define
$$
R:=\Big\{r\in L^1(\R;[0,1])~:~\int_\R r(x)dx=\Mtot \text{ and $\spt(r)$ is compact}\Big\},
$$
where $\spt(r)$ denotes the support of any function $r$, and
$$
Y_n:=\Big\{\mathbf{y}\in\R^n~:~\yii-\yi\geq\ell_n, \quad \forall i=1,\ldots,n-1\Big\}.
$$

We also introduce the  operators $E_n:R\to Y_n$ and $C_n:Y_n\to R$, defined respectively as
\begin{equation}\label{def:En}
E_n[r(\cdot)]:=\mathbf{y}=
\left\{
\begin{array}{l}
y_n=\max(\spt(r)), \\
y_i=\max\Big\{z\in\R~:~\displaystyle\int_z^{\yii}r(x)dx=\ell_n\Big\}, \quad i=n-1,\ldots,2,1
\end{array}
\right.
\end{equation}
and
\begin{equation}\label{def:Cn}
C_n[\mathbf{y}]:=\sum_{i=1}^{n-1}\frac{\ell_n}{\yii(t)-\yi(t)}\chi_{[\yi,\yii)},
\end{equation}
where $\chi_I$ is the indicator function of any subset $I\subset\R$. 
The discretization operator $E_n$ acts on a macroscopic density $\rho(t,\cdot)$, providing a vector of positions $\mathbf{y}(t)=E_n[\rho(t,\cdot)]$ whose components partition the support of the density into segments on which $\rho$ has fixed integral $\ell_n$, see Fig.\ \ref{fig:En_and_Cn}(left).
On the contrary, the operator $C_n$ antidiscretizes a microscopic vector of positions $\mathbf{y}$, giving a \textit{piecewise constant} density $\rho_n(t,\cdot)=C_n[\mathbf{y}(t)]$, see Fig.\ \ref{fig:En_and_Cn}(right).
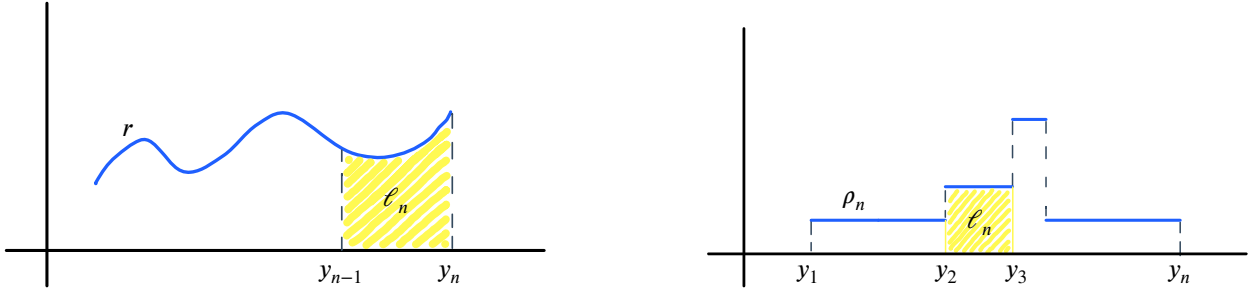
\begin{figure}[h!]
\centering
\begin{overpic}[width=0.4\textwidth]{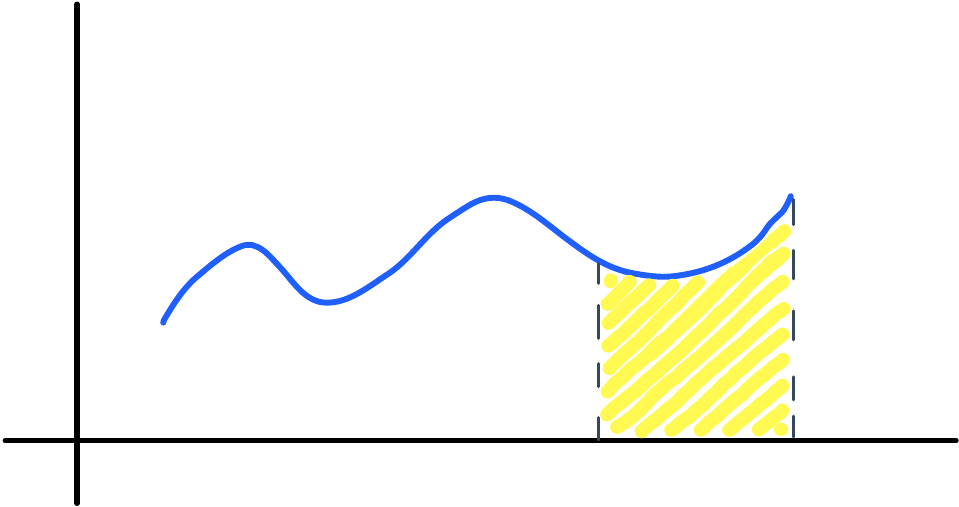}
\put(58,2){$y_{n-1}$}
\put(80,2){$y_n$}
\put(70,15){$\ell_n$}
\put(22,28){$r$}
\end{overpic}\qquad\qquad\qquad
\begin{overpic}[width=0.4\textwidth]{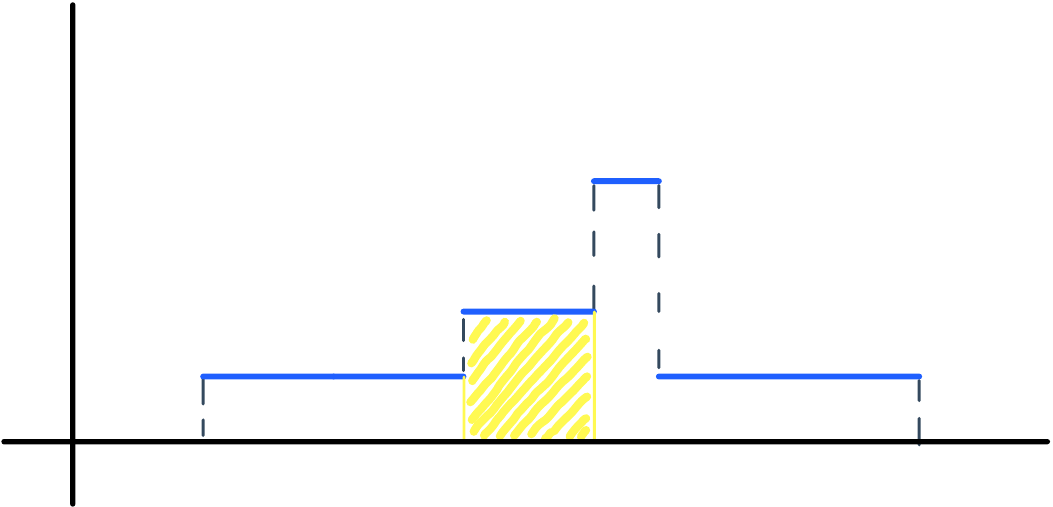}
\put(16.5,2){$y_1$}
\put(42,2){$y_2$}
\put(55,2){$y_3$}
\put(85,2){$y_n$}
\put(48,11){$\ell_n$}
\put(25,15.5){$\rho_n$}
\end{overpic}
\caption{Operators $E_n$ (left) and $C_n$ (right).}
\label{fig:En_and_Cn}
\end{figure}

We are now ready to state the main result about the convergence of the microscopic model (\ref{FTL_1road}) to the macroscopic model (\ref{LWR_1road}). The proof can be found in Di Francesco and Rosini\cite{difrancesco2015ARMA} (see also Colombo and Rossi\cite{colombo2014RSMUP}).

\begin{theorem}\label{teo:convergence_singleroad}
Let (\ref{def:w}) and (\ref{properties_of_v}) hold true and assume
\begin{equation}\label{relazione v-w}
w(\Delta)=v\left(\frac{\ell_n}{\Delta}\right),\qquad \ell_n\leq\Delta<+\infty.
\end{equation} 
Choose $\bar\rho\in R~\cap~BV(\R;[0,1])$ and set $\bar{\mathbf y}=E_n[\bar\rho]$. 
Let $\mathbf y(\cdot)$ be the solution to (\ref{FTL_1road}) with initial condition $\mathbf y(0)=\bar{\mathbf y}$.
Define $\rho_n(t,\cdot)=C_n[\mathbf y(t)]$. Then $\rho_n$ converges almost everywhere and in $ L^1_{\textup{loc}}([0,+\infty)\times\R)$ to the unique entropy solution $\rho$ to the problem (\ref{LWR_1road}) with initial condition $\rho(0,\cdot)=\bar\rho$.
\end{theorem}

\begin{remark} 
Relationship \eqref{relazione v-w} between the velocities makes the link between the microscopic and the macroscopic model.
Assumption $v'<0$ in \eqref{properties_of_v} assures that the function $w$ is increasing, as expected.
\end{remark}

\begin{remark}\label{rem:nosmash} 
Under the same assumptions of Theorem \ref{teo:convergence_singleroad}, it is guaranteed\cite[Lemma 1]{difrancesco2015ARMA} that 
\begin{equation}\label{nonoverlappingcondition}
\yii(t) - \yi(t)\geq\ell_n, \qquad i=1,\ldots,n-1, \quad t>0,
\end{equation}
provided the same non-overlapping condition holds at time $t=0$.
\end{remark}

\subsection{Many-particle limit on networks}\label{sec:limit_network}
Let us consider the case of a road network. We define a road network $\mathcal N$ as a directed graph 
consisting of a finite set of vertexes (junctions) and a set of oriented edges (roads) connecting the vertexes. 
Edges have a positive length and they are provided with a coordinate system.
We assume that for each junction $\textsc{j}$ , there exist disjoint subsets $\textsc{inc}(\textsc{j}), \textsc{out}(\textsc{j})$, representing, respectively, the incoming roads to $\textsc{j}$ and the outgoing roads from $\textsc{j}$. Among junctions, we distinguish two particular subsets consisting of \textit{origins} $\mathcal O$, which are the junctions $\textsc{j}$ such that $\textsc{inc}(\textsc{j})=\emptyset$, and \textit{destinations} $\mathcal D$, which are the junctions $\textsc{j}$ such that $\textsc{out}(\textsc{j})=\emptyset$.

To our purposes, it is convenient to understand the network as an union of paths. More precisely, let us consider all the possible paths on $\mathcal N$ one can obtain joining all the origin nodes in $\mathcal O$ with all the destination nodes in $\mathcal D$, see Fig.\ \ref{fig:paths}. 
\begin{figure}[h!]
\begin{center}
\begin{overpic}
[width=0.49\textwidth]{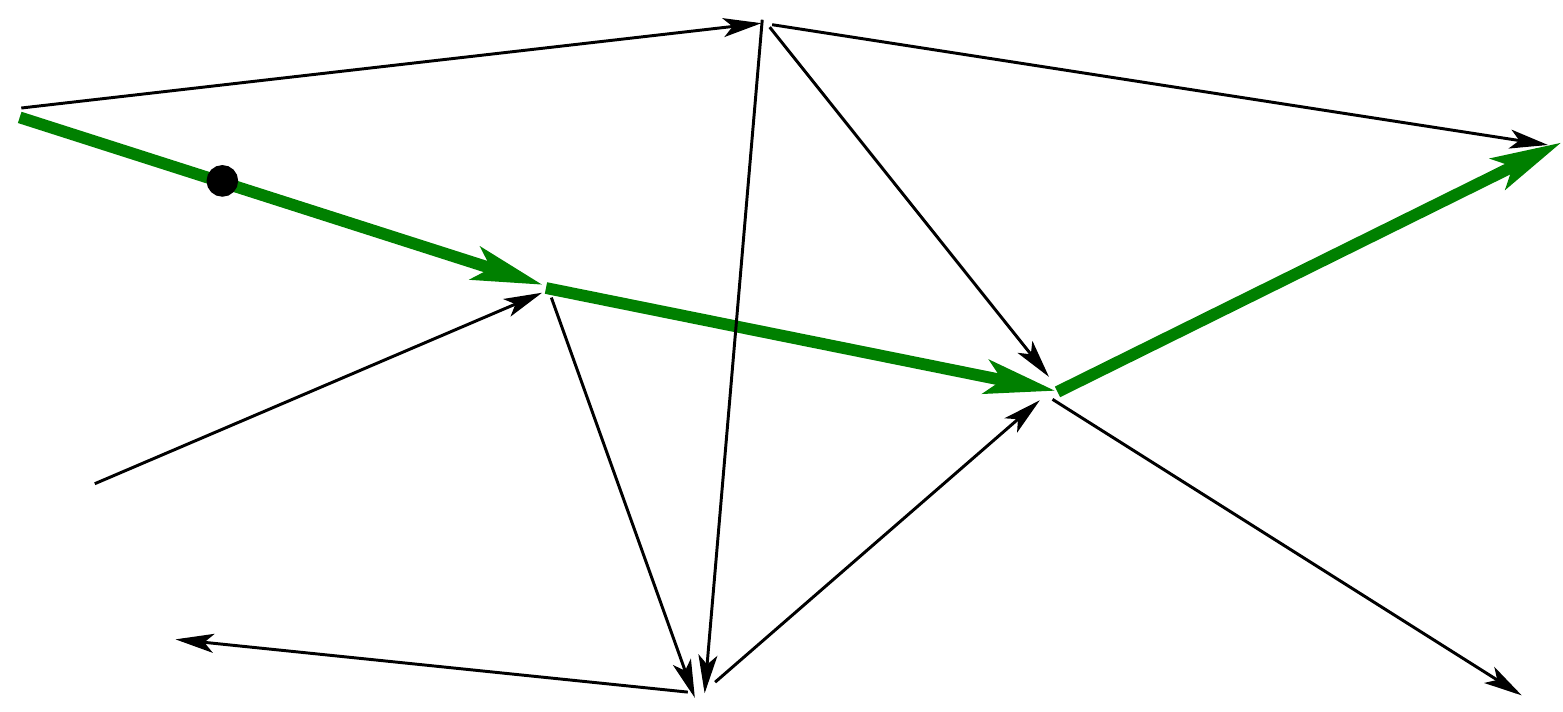}
\put(11,30){\footnotesize $(i,\alpha)$}
\end{overpic}
\begin{overpic}
[width=0.49\textwidth]{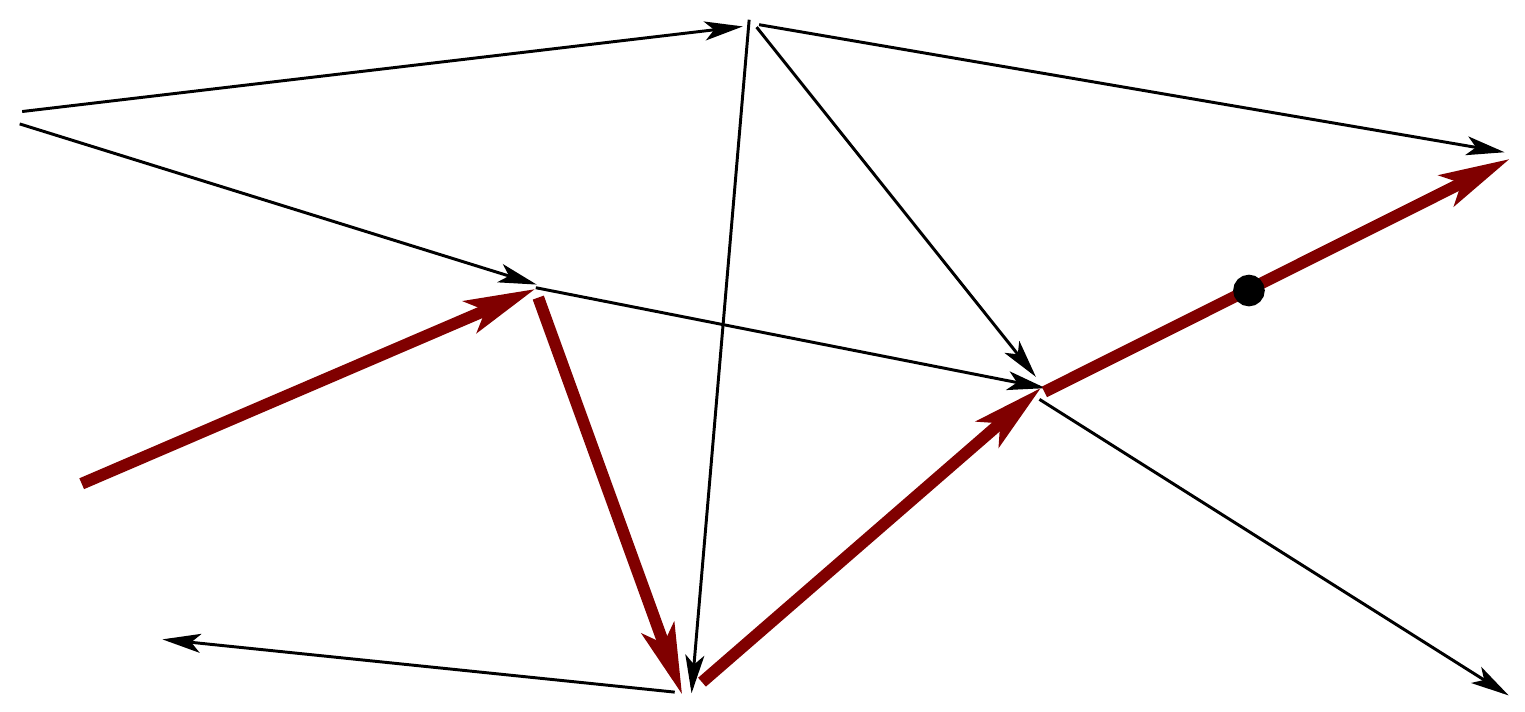}
\put(80,23){\footnotesize $(j,\beta)$}
\end{overpic}
\end{center}
\caption{A generic network $\mathcal N$. Two possible paths are highlighted. Note that the two paths share an arc of the network.}
\label{fig:paths}
\end{figure}
Each path is considered as a \emph{single uninterrupted road}, with no junctions (except for its origin and destination). 
Note that paths can share some arcs of the network. Finally, let us denote by $P$ the total number of paths.

\subsubsection{FtL model}\label{sec:FtLnetwork}
Let us divide the $n$ vehicles in $P$ populations, on the basis of the path they are following. 
Let us denote by $n_\alpha$, $\alpha=1,\ldots,P$, the number of vehicles following path $\alpha$ (we have $\sum_{\alpha=1}^{P}n_\alpha=n$), and label \emph{univocally} all vehicles by the multi-index $(i,\alpha)$, $i=1,\ldots,n_\alpha$, \ $\alpha=1,\ldots,P$. 
Let us also denote by $\yip(t)$ the position of the vehicle $i$ of population $\alpha$ at time $t$, defined as the distance from the origin of the path $\alpha$ (not from the origin of the current road). Since paths can overlap, a vehicle $(j,\beta)$ belonging to a population $\beta\neq \alpha$ with position $y_{j,\beta}(t)$ can be found along path $\alpha$ at time $t$. 
Let us denote by $c_\alpha(j,\beta,t)$ the distance between the vehicle $(j,\beta)$ at time $t$ and the origin of path $\alpha$, along path $\alpha$. In other words $c_\alpha$ acts as a change of coordinates between any path $\beta\neq \alpha$ intersecting path $\alpha$ and the path $\alpha$. The function $c_\alpha$ is trivially extended to vehicles already belonging to population $\alpha$ simply setting $c_\alpha(i,\alpha,t)=y_{i,\alpha}(t)$.

In such a path-based framework the vehicle $i+1$ is no longer necessarily in front of vehicle $i$. 
This force us to introduce a new notation to refer to the vehicle in front. 
First, we define $\next(i,\alpha,t)$ as the multi-index of the nearest vehicle (any population) ahead of the vehicle $(i,\alpha)$ along the path $\alpha$ at time $t$ (setting it to $\emptyset$ if there is no such a vehicle). 
Then we define, for any nonleader vehicle of any population,
$$
d_{i,\alpha}(t):=c_\alpha(\next(i,\alpha,t),t)-c_\alpha(i,\alpha,t).
$$
It is important to note that $d_{i,\alpha}$ is in general discontinuous because at any time other vehicles can join or leave the $i$-th vehicle's path, thus reducing or increasing abruptly the distance $d_{i,\alpha}$ (or even changing the leader/follower status of the $i$-th vehicle). 
An immediate consequence of this fact is that the non-overlapping condition \eqref{nonoverlappingcondition} is no longer guaranteed, even if it holds at the initial time $t=0$. 
Nevertheless, the number of overlapping vehicles is bounded by $\max_{\textsc{j}} |\textsc{inc}(\textsc{j})|$ which is finite (and usually very low), therefore this issue has no impact in the many-particle limit (see Cristiani and Sahu\cite{cristiani2016NHM} for details).

In order to handle this issue we extend the function $w$ defined in (\ref{def:w}) by means of a new function $w^*$ defined as
\begin{equation}\label{def:w*}
w^*:[0,+\infty)\to[0,\wmax], \qquad 
w^*(\Delta)=\left\{
\begin{array}{ll}
w(\Delta), & \text{if } \Delta\geq\ell_n\\
0, & \text{if } \Delta\leq\ell_n 
\end{array}
\right.,
\end{equation}
i.e.\ we assume that (smashed) vehicles closer than $\ell_n$ from the vehicle in front stops completely until the vehicle in front leaves a space $>\ell_n$, then they re-start moving normally. 

We are now ready to introduce the model, described by the following system of $P$ coupled systems of ODEs with discontinuous right-hand side (see Cristiani and Sahu\cite{cristiani2016NHM} for details)
\begin{equation}\label{FTL_network}
\left\{
\begin{array}{ll}
\dot y_{i,\alpha}(t)=w^*(d_{i,\alpha}(t)), & \text{if } \next(i,\alpha,t)\neq\emptyset \\
\dot y_{i,\alpha}(t)=\wmax, & \text{if } \next(i,\alpha,t)=\emptyset
\end{array}
\right.,
\qquad i=1,\ldots,n_\alpha,
\end{equation}
for any $\alpha=1,\ldots,P$.  In the following we will denote by $\mathbf y^\alpha$ the vector $(y_{1,\alpha},\ldots,y_{n_\alpha,\alpha})$ for all $\alpha=1,\ldots,P$, and by $\mathbf y$ the vector $(\mathbf y^1,\ldots,\mathbf y^P)$.

\paragraph{Numerical approximation}
System \eqref{FTL_network} will be numerically integrated by means of the explicit Euler scheme, duly extended to manage the junctions as in Cristiani and Sahu.\cite[Sect.\ 5.1]{cristiani2016NHM}

\subsubsection{LWR model}\label{sec:LWRnetwork}
Many methods were proposed in the literature to extend the LWR model on networks. This task is not completely trivial since the conservation of the mass is not sufficient alone to characterize a unique solution at junctions. As a consequence,  macroscopic traffic models on networks are in general ill-posed unless additional conditions are imposed.
A complete introduction to the field can be found in the book by Garavello and Piccoli.\cite{piccolibook} Beside the classical method\cite{coclite2005SIMA} based on the maximization of the flux at junctions, let us also mention the source-destination model,\cite{garavello2005CMS} the buffer models,\cite{garavello2012DCDS-A,garavello2013bookchapt,herty2009NHM} and the path-based model.\cite{briani2014NHM,bretti2014DCDS-S} 
The last one is based on the path-based interpretation of the network introduced above and consists of a system of $P$ conservation laws with discontinuous flux
\begin{equation}\label{LWR_network}
\partial_t\eta^\alpha + \partial_x\left(\eta^\alpha v^*(\rho)\right)=0, \qquad t>0, \quad x\in\R \quad (\alpha=1,\ldots,P),
\end{equation}
where $\eta^\alpha$ is the density of vehicles following path $\alpha$ and $\rho$ is the total density (all vehicles). For coherence with equation \eqref{def:w*}, we employ here the extended velocity $v^*$ defined as 
\begin{equation}\label{def:v*}
v^*:[0,+\infty)\to[0,\vmax],\qquad 
v^*(r)=
\left\{
\begin{array}{ll}
v(r), & \text{if } r\leq 1 \\
0, & \text{if } r\geq 1.
\end{array}
\right.
\end{equation}

\paragraph{Numerical approximation}
The system of conservation laws \eqref{LWR_network} will be numerically solved by means of a Godunov-based scheme introduced by Bretti et al.,\cite{bretti2014DCDS-S} Briani and Cristiani,\cite{briani2014NHM} and successively used in other related papers.\cite{cristiani2016NHM,briani2018CMS}

\subsubsection{Relation between scales}\label{sec:relationbetweenscales_network}
Recalling the definitions given in Sect.\ \ref{sec:relationbetweenscales_singleroad}, we can state the main result about the convergence of the microscopic model (\ref{FTL_network}) to the macroscopic model (\ref{LWR_network}). The proof can be found in Cristiani and Sahu.\cite{cristiani2016NHM}

\begin{theorem}
Let \eqref{def:w}, \eqref{properties_of_v} and \eqref{relazione v-w} hold true, with the extensions \eqref{def:w*} and \eqref{def:v*}. 
Fix $\alpha\in\{1,\ldots,P\}$, choose $\bar\eta^\alpha\in R~\cap~BV(\R;[0,1])$ and set $\bar{\mathbf{y}}^\alpha=E_n[\bar\eta^\alpha]$. 
Let $\mathbf{y}^\alpha(\cdot)$ be the solution to (\ref{FTL_network}) with initial condition $\mathbf y^\alpha(0)=\bar{\mathbf{y}}^\alpha$ (assuming the other solutions $\mathbf{y}^\beta(\cdot)$, $\beta \neq \alpha$, are given).
Define $\eta_n^\alpha(t,\cdot)=C_n[\mathbf{y}^\alpha]$. 
Then, $\eta^\alpha(t,x):=\lim_{n\to +\infty}\eta_n^\alpha(t,x)$ is a weak solution to (\ref{LWR_network}) (assuming the other solutions $\eta^\beta$, $\beta\neq \alpha$ are given) with initial condition $\bar\eta^\alpha$.
\end{theorem}

Note that the correspondence in the limit is only proved \emph{path by path}, i.e.\ fixing $\alpha$ and assuming the other solutions to be given. This is different from considering the convergence of the fully coupled system of ODEs  to the system of PDEs, which is, to our knowledge, still an open problem.

\subsection{Wasserstein distance}\label{sec:WassersteinDistance}
Let us denote by $(X,\mathfrak{D})$ a complete and separable metric space with distance $\mathfrak{D}$, and by $\mathcal B(X)$ a Borel $\sigma$-algebra of $(X,\mathfrak{D})$. 
Let us also denote by $\mathcal Q^+(X)$ the set of non-negative finite Radon measures on $(X,\mathcal B(X))$. 
Let $\muuno$ and $\mudue$ be two Radon measures in $\mathcal Q^+(X)$ such that $\muuno(X)=\mudue(X)$.
\begin{definition}[Wasserstein distance] 
For any $p\in[1,+\infty)$, the $L^p$-Wasserstein distance between $\muuno$ and $\mudue$ is
\begin{equation}\label{def:W}
W_p(\muuno,\mudue):=\left(\inf_{\gamma\in\Gamma(\muuno,\mudue)}\int_{X\times X}\mathfrak{D}(x,y)^p\ d\gamma(x,y)\right)^{1/p}
\end{equation}
where $\Gamma$ is the set of transport plans connecting $\muuno$ to $\mudue$, i.e.
\begin{equation*}
\Gamma(\muuno,\mudue):=\left\{\gamma\in\mathcal Q^+(X\times X)\text{ s.t. }\gamma(A\times X)=\muuno(A),\ \gamma(X\times B)=\mudue(B),\ \forall \ A,B\subset X\right\}.
\end{equation*}
\end{definition}
\begin{remark}
The measures $\muuno$ and $\mudue$ can be either absolutely continuous or singular with respect to the Lebesgue measure.
\end{remark}
It is well known that the notion of Wasserstein distance can be put in relation with the Monge--Kantorovich optimal mass transfer problem:\cite{villani2009book} a pile of, say, soil, has to be moved to an excavation with same total volume. Moving a unit
quantity of mass has a cost which equals the distance between the source and the
destination point. We consequently are looking for a way to rearrange the first mass
onto the second which requires minimum cost.

In the particular case $X=\R^1$ and $\mathfrak D(x,y)=|x-y|$ various characterizations give alternative, more manageable, definitions. For example, in the case of two Dirac delta functions $\delta(x)$, $\delta(y)$ we simply have $W_p(\delta(x),\delta(y))=|y-x|$, while in the case of two absolutely continuous measures with $d\muuno=\rhouno dx$ and $d\mudue=\rhodue dx$, we have\cite[Rem.\ 2.19]{villani2003book}
\begin{equation}\label{caratterizzazioneW1inR}
W_1(\muuno,\mudue)=\int_{\R}|\Funo(x)-\Fdue(x)|dx=\int_{\R}|(\Funo)^{-1}(x)-(\Fdue)^{-1}(x)|dx,\quad \text{with} \quad \Funo(x):=\int_{-\infty}^{x}\rhouno(z)dz, \quad \Fdue(x):=\int_{-\infty}^{x}\rhodue(z)dz
\end{equation}
and, for $p\geq 1$,
\begin{equation}\label{caratterizzazioneWpinR}
W_p(\muuno,\mudue)=\left(\int_{\R}|T^*(x)-x|^p\ \rhouno(x) dx\right)^{1/p},
\quad\text{ where \ } T^*:\R\to\R \text{ satisfies }\quad
\int_{-\infty}^{x} \rhouno(z)dz=
\int_{-\infty}^{T^*(x)} \rhodue(z)dz
\quad \forall x\in\mathbb R.
\end{equation}
Equation \eqref{caratterizzazioneWpinR} translates the fundamental property of \emph{monotone rearrangement} of the mass,\cite[Rem.\ 2.19]{villani2003book} i.e.\ the optimal strategy consists in transferring the mass starting from the left, see Fig.\ \ref{fig:monotonerearrangement}.
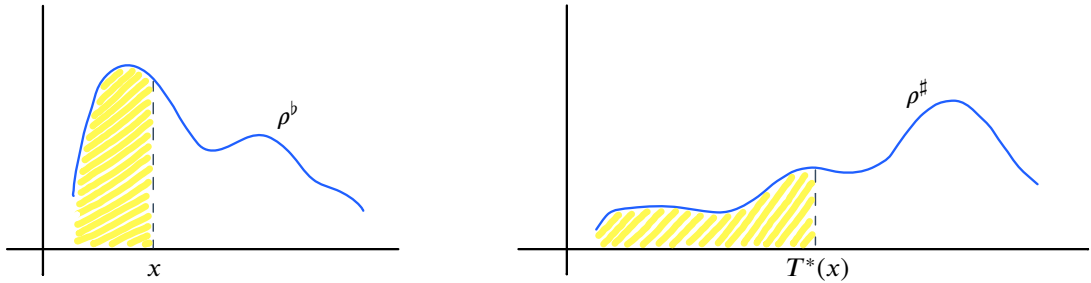
\begin{figure}[h!]
\centering
\begin{overpic}[width=0.8\textwidth]{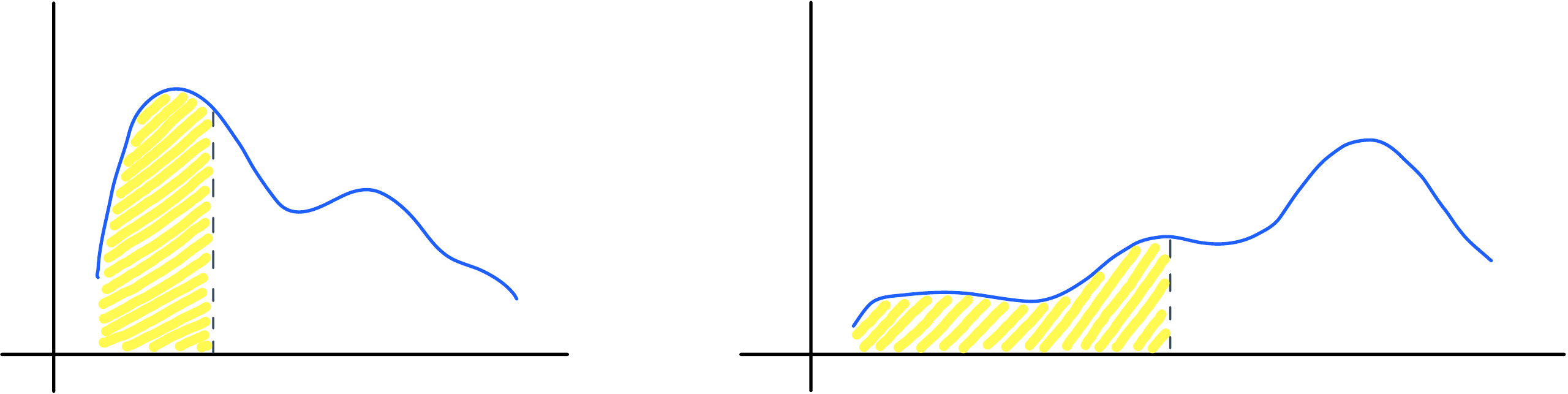}
\put(13,0){$x$}
\put(72,0){$T^*(x)$}
\put(25,14){$\rhouno$}
\put(83,16){$\rhodue$}
\end{overpic}
\caption{Property of monotone rearrangement \eqref{caratterizzazioneWpinR}.}
\label{fig:monotonerearrangement}
\end{figure}

\paragraph*{Numerical approximation} 
Numerical approximation of the Wasserstein distance is easy in the particular case $X=\R^1$ and $\mathfrak D(x,y)=|x-y|$, since quadrature formulas can be easily employed in \eqref{caratterizzazioneW1inR} or \eqref{caratterizzazioneWpinR}.

The case of more general spaces, including networks embedded in $\R^2$, is more difficult. Typically the problem is solved by reformulating it in a fully discrete setting as a Linear Programming problem. The interested reader can find in the books by Santambrogio\cite[Sec.\ 6.4.1]{santambrogiobook} and Sinha\cite[Chap.\ 19]{sinhabook}  the complete procedure, recently used also in the context of traffic flow.\cite{briani2018CMS}
More sophisticated characterizations based on a fluid-dynamic approach,\cite{benamou2000NM} $p$-Laplacian,\cite{mazon2015SIOPT} or a variational approach\cite{mazon2015SIOPT} are also available. 

\section{Comparing traffic states in macroscopic and microscopic models}\label{sec:compare}
In this paper we are mainly concerned with the measure of the ``distance'' between two traffic states, i.e.\ we want to quantify to what extent two traffic states are close to each other. For example, if we run a traffic model in a real application, we would like to measure how close the solution of the model is to the measured traffic state. 

In the rest of the paper we will put ourself in the context of \emph{sensitivity analysis}, broadly following Briani et al.\cite{briani2018CMS} 
Let us assume to have run both the microscopic and macroscopic model twice, with different initial conditions or model parameters. Doing so, at final time $t=t_f$ we end up with two different traffic scenarios to be compared. In the microscopic case we will have two vectors $\ybfuno(t_f)$ and $\ybfdue(t_f)$, while in the macroscopic case we will have two density distributions $\rhouno(t_f,\cdot)$ and $\rhodue(t_f,\cdot)$.

Before specializing the discussion on a specific scale, let us introduce a suitable notion of distance on the road network. We will denote by $\DG(a,b)$, $a,b\in\mathcal N$, the distance between two points $a$ and $b$ on the network $\mathcal N$, defined as the length of the shortest path (in the Euclidean sense) on $\mathcal N$ joining $a$ and $b$. 
Note that the distance $\DG$ can be numerically computed by means of, e.g., the Dijkstra algorithm.\cite{dijkstra1959NM}

\subsection{Comparing vehicle positions}\label{sec:compareFtLsolutions}
In the microscopic framework each vehicle is uniquely labeled, and its position is known exactly. This means that we can compare two traffic scenarios (vehicle arrangements) comparing the position of each vehicle in each scenario.

That said, it is natural to define the distance between any two microscopic traffic states $\ybfuno, \ybfdue\in\R^n$ as
\begin{equation}\label{def:Dftl-singleroad}
\Dftl(\ybfuno,\ybfdue):=\left(\ell_n\sum_{i=1}^n |\yuno_i-\ydue_i|^p\right)^{1/p} \quad \text{(on a single road),}
\end{equation}
\begin{equation}\label{def:Dftl-network}
\Dftl(\ybfuno,\ybfdue):=\left(\ell_n\sum_{(i,\alpha)} \DG(\yuno_{i,\alpha},\ydue_{i,\alpha})^p\right)^{1/p} \quad \text{(on network $\mathcal N$)}.
\end{equation}
These definitions translate the fact that the two traffic states are regarded as ``close'' if each vehicle is approximately in the same position in the two scenarios. The distance is then proportional to the effort one should make in order to move each vehicle from the position occupied in the run $\source$ to the position occupied in the run $\dest$. 
Note that the movement can only happen along the network, but it must not necessarily respect road rules like road directions, priorities, traffic lights, etc.; cf.\ Briani et al.\cite{briani2018CMS}

It is useful to note that the Wasserstein distance can also be used to compute the distance between vectors $\ybfuno$ and $\ybfdue$, obtaining 
\begin{equation}\label{def:Wmicro-singleroad}
W_p(\pi[\ybfuno],\pi[\ybfdue]),
\quad \text{with} \quad 
\pi[\mathbf y]:=\ell_n\sum_{i=1}^n \delta(y_i) 
\quad \text{and} \quad 
\mathfrak{D}(x,y)=|x-y| \quad \text{(on a single road),}
\end{equation}
\begin{equation}\label{def:Wmicro-network}
W_p(\pi[\ybfuno],\pi[\ybfdue]),
\quad \text{with} \quad 
\pi[\mathbf y]:=\ell_n\sum_{(i,\alpha)} \delta(y_{i,\alpha}) 
\quad \text{and} \quad 
\mathfrak{D}(x,y)=\mathcal D_\mathcal N(x,y) \quad \text{(on network $\mathcal N$)}.
\end{equation}
Doing this, we implicitly allow the possibility of exchange vehicles from one scenario to another since the optimal transport plan $\gamma^*\in\Gamma$ not necessarily moves vehicle $i$ (resp., $(i,\alpha)$) into vehicle $i$ (resp., $(i,\alpha)$) as we assume instead in \eqref{def:Dftl-singleroad}-\eqref{def:Dftl-network}. This is is easily confirmed by the following characterization\cite[Introduction]{villani2003book} valid in the discrete case,
$$
W_p(\pi[\ybfuno],\pi[\ybfdue])=\inf_{\sigma\in\mathcal S_n}\left\{\ell_n\sum_{i=1}^n|\yuno_i-\ydue_{\sigma(i)}|\right\}\quad \text{(on a single road)},
$$
$$
W_p(\pi[\ybfuno],\pi[\ybfdue])=\inf_{\sigma\in\mathcal S_n}\left\{\ell_n\sum_{i=1}^n\mathcal D_\mathcal N(\yuno_i,\ydue_{\sigma(i)})\right\}\quad \text{(on network $\mathcal N$)},
$$
where $\mathcal{S}_n$ is the set of permutation in $\{1,2,\ldots,n\}$.
We will show in Sect.\ \ref{sec:singleroadanalysis} that the two distances defined in \eqref{def:Dftl-singleroad} and \eqref{def:Wmicro-singleroad} actually coincide in the framework of traffic flow models considered in this paper, but this is not the case, in general, for the distances \eqref{def:Dftl-network} and \eqref{def:Wmicro-network}. 

\subsection{Comparing traffic densities}\label{sec:compareLWRsolutions}
In the macroscopic framework we are not able to distinguish single vehicles, then we have to compare vehicle densities \emph{as a whole}. To do that, we can follow in a broad sense the same ideas described in the microscopic framework, quantifying the effort one has to make in order to rearrange all the mass from the first scenario (mass density distribution $\rhouno$) to the second one (mass density distribution $\rhodue$). Since there are infinite ways to rearrange the mass moving an elementary particle of mass $\rho(x)dx$ at a time, one should look at the ``cheapest'' one, in terms of a cost function to be suitably defined.

Contrary to the microscopic setting, in the macroscopic setting the Wasserstein distance is exactly what we are looking for, thanks to the well known relation with the optimal mass transfer problem, see Section \ref{sec:WassersteinDistance}. Therefore, we simply define
\begin{equation}\label{def:Dlwr}
\Dlwr(\rhouno,\rhodue):=W_p(\nu[\rhouno],\nu[\rhodue]),
\end{equation}
where the measure $\nu[\rho]$, for $\rho\in L^1$, is the unique measure such that 
\begin{equation}\label{def:nu[rho]}
\nu[\rho](A):=\int_A \rho dx,\qquad \forall A\subseteq X,
\end{equation} 
where either $X=\R$, $\mathfrak{D}(x,y)=|x-y|$ (single road) or $X=\mathcal N$, $\mathfrak{D}(x,y)=\DG(x,y)$ (road network).

\section{Comparing comparisons on a single road}\label{sec:singleroadanalysis}
Let us now consider the following two pairs of Cauchy problems
\begin{equation}\label{Cauchyproblems_micro}
(\source)\quad
\left\{
\begin{array}{ll}
\dot y^\source_i=\wuno(\yuno_{i+1}-\yuno_i),\qquad i=1,\ldots,n-1 \\
\dot y^\source_n=\vmaxuno \\
\ybfuno(0)=\ybfbaruno
\end{array}
\right.\ ,
\qquad\qquad
(\dest)\quad
\left\{
\begin{array}{ll}
\dot y^\dest_i=\wdue(\ydue_{i+1}-\ydue_i),\qquad i=1,\ldots,n-1 \\
\dot y^\dest_n=\vmaxdue \\
\ybfdue(0)=\ybfbardue
\end{array}
\right.\ ,
\end{equation}
\vskip0.2cm
\begin{equation}\label{Cauchyproblems_macro}
(\source)\quad
\left\{
\begin{array}{ll}
\partial_t\rhouno+\partial_x(\rhouno \vuno(\rhouno))=0 \\
\rhouno(0,\cdot)=\rhobaruno
\end{array}
\right.\ , \phantom{XXXXXXXx}
\qquad\qquad
(\dest)\quad
\left\{
\begin{array}{ll}
\partial_t\rhodue+\partial_x(\rhodue \vdue(\rhodue))=0 \\
\rhodue(0,\cdot)=\rhobardue
\end{array}
\right.\ , \phantom{XXXXXXXx}
\end{equation}
where $\vuno(\rhouno)=\vmaxuno(1-\rhouno)$ and $\vdue(\rhodue)=\vmaxdue(1-\rhodue)$ for some constants $\vmaxuno, \vmaxdue>0$, cf.\ \eqref{v(rho)}. Assume also that \eqref{def:w},\eqref{properties_of_v},\eqref{relazione v-w} hold true for both pairs $(\vuno,\wuno)$ and $(\vdue,\wdue)$.

\subsection{Analytical results}\label{sec:singleroad_analyticalresults}
We have the following result.
\begin{theorem}\label{teo:w->W_singleroad}
Consider the problems \eqref{Cauchyproblems_micro}-\eqref{Cauchyproblems_macro} and fix a final time $t_f$ so that $t<t_f$.
Choose the initial data $\rhobaruno, \rhobardue\in R\cap BV(\R;[0,1])$ such that 
\begin{equation}\label{ipotesinuova}
\begin{split}
\rhobaruno(x)>\rhomin \text{\ \ for a.e.\ } x\in(\inf \spt(\rhobaruno),\sup\spt(\rhobaruno)),\\
\rhobardue(x)>\rhomin \text{\ \ for a.e.\ } x\in(\inf \spt(\rhobardue),\sup\spt(\rhobardue)).
\end{split}
\end{equation} 
for some $\rhomin>0$. 
Set $\ybfbaruno=E_n[\rhobaruno]$, $\ybfbardue=E_n[\rhobardue]$.
Let $\ybfuno(\cdot)$ and $\ybfdue(\cdot)$ be the solutions to Cauchy problems $(\source)$ and $(\dest)$ in \eqref{Cauchyproblems_micro}, respectively.
Define $\rhouno_n(t,\cdot)=C_n[\ybfuno(t)]$ and $\rhodue_n(t,\cdot)=C_n[\ybfdue(t)]$. Then
\begin{itemize}
\item[(i)] $\Dftl(\ybfuno(t),\ybfdue(t))=W_p(\pi[\ybfuno(t)],\pi[\ybfdue(t)]), \quad t<t_f$;
\item[(ii)] $\lim_{n\to+\infty} \Dftl(\ybfuno(t),\ybfdue(t))=\Dlwr(\rhouno(t,\cdot),\rhodue(t,\cdot)),\quad t<t_f,\quad$ where
\begin{equation}\label{def:rholimite}
\begin{split}
\rhouno(t,x):=\lim_{n\to+\infty}\rhouno_n(t,x),\quad \text{for a.e.\ } x\in\R, \quad t<t_f,\\
\rhodue(t,x):=\lim_{n\to+\infty}\rhodue_n(t,x),\quad \text{for a.e.\ } x\in\R, \quad t<t_f,
\end{split}
\end{equation}
are the entropy solutions to Cauchy problems $(\source)$ and $(\dest)$ in \eqref{Cauchyproblems_macro}, respectively.
\end{itemize}
\end{theorem}

\begin{proof}
For the sake of readability, from now on we will omit the dependence on time $t$. 

Let us first recall here the crucial fact that the FtL model does not allow overtaking, thus, for any $t<t_f$, we have $\yuno_i<\yuno_{i+1}$ and $\ydue_i<\ydue_{i+1}$ $\forall i=1,\ldots,n-1$, see Remark \ref{rem:nosmash}.

Let us also note that, since $\rhobaruno, \rhobardue\in R$, we have $\int_\R \rhobaruno dx=\int_\R \rhobardue dx=\Mtot$. 
This implies that 
\begin{equation}\label{masseugualiognin}
\int_{\R}\rhouno_n dx=\int_{\R}\rhodue_n dx=\Mtot \qquad \forall n
\end{equation}
because $\rhouno_n$ and $\rhodue_n$ are defined by two vectors $\ybfuno$, $\ybfdue$ of the same size through the operators $E_n$ and $C_n$. 

(i) Let $\nuuno_n:=\nu[\rhouno_n]$, $\nudue_n:=\nu[\rhodue_n]$ be the measures defined as in \eqref{def:nu[rho]}. 
By definition, they are absolutely continuous with respect to the Lebesgue measure. Moreover, by \eqref{masseugualiognin} we have $\int_{\R}d\nuuno_n=\int_{\R}d\nuuno_n=\Mtot$.
Therefore it applies the property of \emph{monotone rearrangement} (see Section \ref{sec:WassersteinDistance}) which states that  the measure $\nuuno_n$ is optimally transported onto $\nudue_n$ by the map $T_n^*:\R\to\R$ such that 
\begin{equation}\label{def:Tn*}
\int_{-\infty}^x d\nuuno_n=\int_{-\infty}^{T_n^*(x)} d\nudue_n, 
\qquad\text{ or, equivalently, }\qquad
\int_{-\infty}^x \rhouno_n(z) dz=\int_{-\infty}^{T_n^*(x)} \rhodue_n(z) dz,\qquad x\in\R.
\end{equation}
Moreover, we have
\begin{equation}\label{caratterizzazione_W_1D}
\Dlwr(\rhouno_n,\rhodue_n)=\left(\int_{\R}|T_n^*(x)-x|^p\ \rhouno_n(x) dx\right)^{1/p}.
\end{equation}

It is useful to note here that \eqref{def:elln}, \eqref{def:Cn} and \eqref{nonoverlappingcondition} imply $\rhouno_n>0$ in the interval $[\yuno_1,\yuno_n)$ and $\rhodue_n>0$ in the interval $[\ydue_1,\ydue_n)$. 
Therefore the function $T^*_n$ is \emph{uniquely determined} in the interval $(\yuno_1,\yuno_n)$, in the sense that for any $x\in(\yuno_1,\yuno_n)$ there exists a unique value of $T^*_n(x)$ such that \eqref{def:Tn*} holds. 
Outside the interval this is no longer true. 
For example, if $x=\yuno_1$ we have $\int_{-\infty}^{\yuno_1} \rhouno_n(z) dz=0$ and $\int_{-\infty}^y \rhodue_1(z) dz=0$ for any $y\leq \ydue_1$. 
Analogously, if $x=\yuno_n$ we have $\int_{-\infty}^{\yuno_n} \rhouno_n(z) dz=\Mtot$ and $\int_{-\infty}^y \rhodue_n(z) dz=\Mtot$ for any $y\geq \ydue_n$. 
For our purposes, it is useful to extend the definition of $T_n^*$ setting 
\begin{equation}\label{extensionTn*}
T_n^*(x)=\ydue_1 \quad \text{for} \quad  x\leq \yuno_1 
\qquad \text{and} \qquad  
T_n^*(x)=\ydue_n \quad  \text{for}\quad  x\geq \yuno_n.
\end{equation} 
Doing this, we also have 
\begin{equation}\label{T*ninspt(rhoduen)}
T^*_n(x)\in\spt(\rhodue_n)=[\ydue_1(0),\ydue_n(0)+\vmax t],\qquad
\forall x\in\R,\quad t<t_f,
\end{equation}
which is a natural fact considering that outside $\spt(\rhodue_n)$ there is no mass to move to. 

\medskip

Setting $x=\yuno_i$ in \eqref{def:Tn*}, we get
$$
\int_{-\infty}^{\yuno_i} \rhouno_n(z) dz=\int_{-\infty}^{T_n^*(\yuno_i)} \rhodue_n(z) dz,\qquad i=1,\ldots,n.
$$
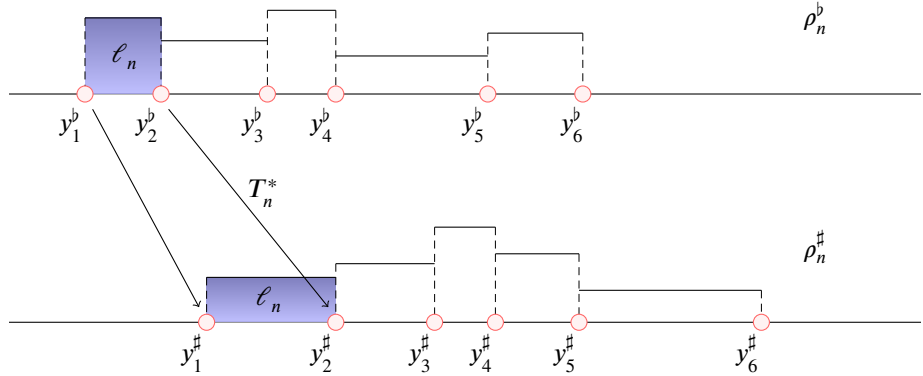
\begin{figure}[h!]
\centering
\begin{tikzpicture}
\draw (1,3) -- (13,3);
\shade [fill=blue, fill opacity=0.5] (2,3) rectangle (3,4);
\draw (2,4) -- (3,4);
\draw (3,3.7) -- (4.4,3.7);
\draw (4.4,4.1) -- (5.3,4.1);
\draw (5.3,3.5) -- (7.3,3.5);
\draw (7.3,3.8) -- (8.55,3.8);
\draw[densely dashed,ultra thin] (2,4) -- (2,3);
\draw[densely dashed,ultra thin] (3,4) -- (3,3);
\draw[densely dashed,ultra thin] (4.4,4.1) -- (4.4,3);
\draw[densely dashed,ultra thin] (5.3,4.1) -- (5.3,3);
\draw[densely dashed,ultra thin] (7.3,3.8) -- (7.3,3);
\draw[densely dashed,ultra thin] (8.55,3.8) -- (8.55,3);
\node[text width=11mm] at (12,4) {$\rhouno_n$};
\filldraw[color=red!60,fill=red!5,thin] (2,3) circle (3pt);
\filldraw[color=red!60,fill=red!5,thin] (3,3) circle (3pt);
\filldraw[color=red!60,fill=red!5,thin] (4.4,3) circle (3pt);
\filldraw[color=red!60,fill=red!5,thin] (5.3,3) circle (3pt);
\filldraw[color=red!60,fill=red!5,thin] (7.3,3) circle (3pt);
\filldraw[color=red!60,fill=red!5,thin] (8.55,3) circle (3pt);
\node[text width=11mm] at (2.9,3.5) {$\ell_n$};
\node[text width=11mm] at (2.2,2.6) {$\yuno_1$};
\node[text width=11mm] at (3.2,2.6) {$\yuno_2$};
\node[text width=11mm] at (4.6,2.6) {$\yuno_3$};
\node[text width=11mm] at (5.5,2.6) {$\yuno_4$};
\node[text width=11mm] at (7.5,2.6) {$\yuno_5$};
\node[text width=11mm] at (8.8,2.6) {$\yuno_6$};
\draw (1,0) -- (13,0);
\shade [fill=blue, fill opacity=0.5] (3.6,0) rectangle (5.3,0.59);
\draw (3.6,0.59) -- (5.3,0.59);
\draw (5.3,0.77) -- (6.6,0.77);
\draw (6.6,1.25) -- (7.4,1.25);
\draw (7.4,0.9) -- (8.5,0.9);
\draw (8.5,0.42) -- (10.9,0.42);
\draw[densely dashed,ultra thin] (3.6,0.59) -- (3.6,0);
\draw[densely dashed,ultra thin] (5.3,0.77) -- (5.3,0);
\draw[densely dashed,ultra thin] (6.6,1.25) -- (6.6,0);
\draw[densely dashed,ultra thin] (7.4,1.25) -- (7.4,0);
\draw[densely dashed,ultra thin] (8.5,0.9) -- (8.5,0);
\draw[densely dashed,ultra thin] (10.9,0.42) -- (10.9,0);
\node[text width=11mm] at (12,1) {$\rhodue_n$};
\filldraw[color=red!60,fill=red!5,thin] (3.6,0) circle (3pt);
\filldraw[color=red!60,fill=red!5,thin] (5.3,0) circle (3pt);
\filldraw[color=red!60,fill=red!5,thin] (6.6,0) circle (3pt);
\filldraw[color=red!60,fill=red!5,thin] (7.4,0) circle (3pt);
\filldraw[color=red!60,fill=red!5,thin] (8.5,0) circle (3pt);
\filldraw[color=red!60,fill=red!5,thin] (10.9,0) circle (3pt);
\node[text width=11mm] at (4.7,1.7) {$T^*_n$};
\node[text width=11mm] at (4.8,0.3) {$\ell_n$};
\node[text width=11mm] at (3.8,-0.4) {$\ydue_1$};
\node[text width=11mm] at (5.5,-0.4) {$\ydue_2$};
\node[text width=11mm] at (6.8,-0.4) {$\ydue_3$};
\node[text width=11mm] at (7.6,-0.4) {$\ydue_4$};
\node[text width=11mm] at (8.7,-0.4) {$\ydue_5$};
\node[text width=11mm] at (11.1,-0.4) {$\ydue_6$};
\draw[->] (2.1,2.8) -- (3.5,0.2);
\draw[->] (3.1,2.8) -- (5.2,0.2);
\end{tikzpicture}
\caption{Optimal mass transport map $T^*_n$ from $\nuuno_n$ into $\nudue_n$.}
\label{fig:riarrangiamentomonotono}
\end{figure}
Using \eqref{def:En} and \eqref{extensionT*}, we immediately get $T_n^*(\yuno_i)=\ydue_i$ for $i=1,\ldots,n$ (see Fig.\ \ref{fig:riarrangiamentomonotono}) and that the mass in $[\yuno_i,\yuno_{i+1})$ is transported into $[\ydue_i,\ydue_{i+1})$ for $i=1,\ldots,n-1$.

Now, to conclude let us focus on the concentrated distributions $\pi[\ybfuno]$ and $\pi[\ybfdue]$. 
We note that, for all $i=1,\ldots,n$, $\spt(\delta_{\yuno_i})=\yuno_i\in[\yuno_i,\yuno_{i+1})$. Therefore the mass in $\yuno_i$ must be optimally moved into the interval $[\ydue_i,\ydue_{i+1})$ (otherwise the previous construction would not be optimal).
Since the arrival mass in $[\ydue_i,\ydue_{i+1})$ is all concentrated in $\ydue_i$, the optimal mass transfer of $\pi[\ybfuno]$ onto $\pi[\ybfdue]$ moves $\delta_{\yuno_i}$ into $\delta_{\ydue_i}$. 
In conclusion, $W_p(\pi[\ybfuno],\pi[\ybfdue])=\left(\ell_n\sum_{i=1}^n|\yuno_i-\ydue_i|^p\right)^{1/p}$ as desired.

\bigskip

(ii) We divide the proof in three steps. 

\medskip

(ii.A) Let us first prove that $\Dlwr(\rhouno,\rhodue)$ is well defined. The existence of functions $\rhouno$, $\rhodue$ in \eqref{def:rholimite} is assured by Theorem \ref{teo:convergence_singleroad}. 
Being both solution to \eqref{LWR_1road}, they are in $BV(\R;[0,1])$\cite[Theor.\ 6.2.6]{dafermosbook} (and then they are a.e.\ continuous).
It remains to show that 
\begin{equation}\label{int rho1=int rho2}
\int_\R \rhouno dx=\int_\R \rhodue dx=\Mtot.
\end{equation} 
By construction, we have  that $\rhouno_n$, $\rhodue_n$ are bounded by 1 for any $x$, $t$, $n$. 
Moreover, 
\begin{equation}\label{rhon_a_supporto_compatto}
\spt(\rhouno_n)\subseteq[\yuno_1(0),\yuno_n(0)+\wmax t_f] \qquad\text{and}\qquad 
\spt(\rhodue_n)\subseteq[\ydue_1(0),\ydue_n(0)+\wmax t_f].
\end{equation}  
Therefore 
the 
dominated convergence theorem allows us to write
\begin{equation}\label{limint=intlim}
\lim_{n\to+\infty}\int_{\R}\rhouno_n dx =\int_\R \rhouno dx, \qquad 
\lim_{n\to+\infty}\int_{\R}\rhodue_n dx =\int_\R \rhodue dx,
\end{equation}
and then by \eqref{def:rholimite} and \eqref{masseugualiognin} we easily get \eqref{int rho1=int rho2}.

\medskip

(ii.B) Applying again the property of monotone rearrangement to $\nu[\rhouno]$ and $\nu[\rhodue]$ we get that there exists a map $T^*:\R\to\R$ such that 
\begin{equation}\label{def:T*}
\int_{-\infty}^x \rhouno(z) dz=\int_{-\infty}^{T^*(x)} \rhodue(z) dz,\qquad x\in\R.
\end{equation}
As in \eqref{caratterizzazione_W_1D}, this is the optimal transport map which allows to compute the distance between the two density functions, i.e.\
\begin{equation}
\Dlwr(\rhouno,\rhodue)=\left(\int_\R|T^*(x)-x|^p\rhouno(x)dx\right)^{1/p}. 
\end{equation}

Note that the functions $\rhouno, \rhodue$, being solutions to \eqref{LWR_1road} with initial data satisfying \eqref{ipotesinuova}, remain strictly positive at any time $t<t_f$ in the interval of interest, i.e.
\begin{equation}\label{conseguenzeipotesinuova}
\begin{split}
\rhouno(x)>0 \text{\ \ for a.e.\ } x\in[\inf \spt(\rhouno),\sup\spt(\rhouno)],\quad t<t_f,\\
\rhodue(x)>0 \text{\ \ for a.e.\ } x\in[\inf \spt(\rhodue),\sup\spt(\rhodue)],\quad t<t_f.
\end{split}
\end{equation} 
Indeed, at the leftmost boundary of the support the solution is that of a Riemann problem with left state $\rho=0$ and right state $\rho\in[\rhomin,1]$, which produces a shock with positive or null velocity. As a consequence, the leftmost boundary of the support remains fixed or moves to the right. 
At the rightmost boundary of the support  the solution is that of a Riemann problem with left state $\rho\in[\rhomin,1]$ and right state $\rho=0$, which produces a rarefaction fan. Again, the rightmost limit of the support is simply moved to the right (with velocity $\frac{d}{d\rho}(\rho v(\rho))|_{\rho=0}=v(0)=\vmax$) and the solution remains strictly positive inside the rarefaction fan. 
Sufficiently far from the extremes of the support, i.e.\ where the two Riemann problems described above do not influence the solution, the result follows by the basic monotonicity property of the solution to conservation laws,\cite[Theor.\ 6.2.3]{dafermosbook} which states that the solution remains bounded at any time by the same values which bound the initial datum. 

An important consequence of \eqref{conseguenzeipotesinuova} is that, as we did in \eqref{extensionTn*}, we can make the function $x\mapsto T^*(x)$ a.e.\ uniquely determined in $\R$ setting 
\begin{equation}\label{extensionT*}
T^*(x)=\inf\spt(\rhodue) \quad \text{for} \quad x\leq\inf\spt(\rhouno) 
\qquad \text{and} \qquad 
T^*(x)=\sup\spt(\rhodue) \quad \text{for} \quad  x\geq\sup\spt(\rhouno).
\end{equation}

\medskip

Now we want to prove that $\lim_{n\to +\infty}T^*_n= T^*$ for a.e.\ $x\in\spt(\rhouno)$. 
Let us first prove that $\{T^*_n\}_n$ is a Cauchy sequence, then the limit exists. After that, we will prove that the limit is indeed a function which satisfies \eqref{def:T*}.

Using \eqref{def:Tn*}, we have
\begin{equation*}
\begin{split}
\int_{T^*_n(x)}^{T^*_m(x)}\rhodue dz=
\int_{-\infty}^{T^*_m(x)}\rhodue dz -
\int_{-\infty}^{T^*_m(x)}\rhodue_m dz +
\int_{-\infty}^{T^*_m(x)}\rhodue_m dz -
\int_{-\infty}^{T^*_n(x)}\rhodue_n dz +
\int_{-\infty}^{T^*_n(x)}\rhodue_n dz -
\int_{-\infty}^{T^*_n(x)}\rhodue dz = \\
\int_{-\infty}^{T^*_m(x)}\big(\rhodue-\rhodue_m\big) dz +
\int_{-\infty}^{x}\rhouno_m dz -
\int_{-\infty}^{x}\rhouno_n dz +
\int_{-\infty}^{T^*_n(x)}\big(\rhodue_n-\rhodue\big) dz = \\
\int_{-\infty}^{T^*_m(x)}\big(\rhodue-\rhodue_m\big) dz +
\int_{-\infty}^{x}\big(\rhouno_m-\rhouno_n\big) dz +
\int_{-\infty}^{T^*_n(x)}\big(\rhodue_n-\rhodue\big) dz.
\end{split}
\end{equation*}

Choosing two real values $\alpha$ and $\beta$ such that the interval $[\alpha,\beta]$ contains the supports of the functions $\rhouno,\rhouno_n,\rhodue,\rhodue_n$ for any $n$ and $t<t_f$, we get
\begin{equation}\label{integraleCorrado}
\left| \ \int_{T^*_n(x)}^{T^*_m(x)}\rhodue dz\right| \leq
\int_{\alpha}^{\beta}|\rhodue-\rhodue_m| dz +
\int_{\alpha}^{\beta}|\rhouno_m-\rhouno_n| dz +
\int_{\alpha}^{\beta}|\rhodue_n-\rhodue| dz.
\end{equation}
Using the convergence result of Theorem \ref{teo:convergence_singleroad} we get that $\left| \int_{T^*_n(x)}^{T^*_m(x)}\rhodue dz\right|\to 0$ for $n,m\to +\infty$.
We also get the set-theoretic convergence of the support $\spt(\rhodue_n)\to\spt(\rhodue)$ for $n\to +\infty$.
Since, by \eqref{T*ninspt(rhoduen)}, we have $T^*_n\in\spt(\rhodue_n)$, $T^*_m\in\spt(\rhodue_m)$, in the limit both $T^*_n$ and $T^*_m$ have to enter the interval $[\inf\spt(\rhodue),\sup\spt(\rhodue)]$, where $\rhodue$ is a.e.\ strictly positive (see \eqref{conseguenzeipotesinuova}). This allows to conclude that $|T^*_n-T^*_m|\to 0$ for $n,m\to +\infty$, otherwise the first integral in \eqref{integraleCorrado} cannot tend to 0.

We have proved that $T^*_n\to S$ for some function $S$.
Finally, using again Theorem \ref{teo:convergence_singleroad} and the compactness of the supports of the density functions, we have
$$
\int_{-\infty}^{x}\rhouno dz  =
\lim_{n\to +\infty}\int_{-\infty}^{x}\rhouno_n dz \stackrel{\eqref{def:Tn*}}{=}
\lim_{n\to +\infty}\int_{-\infty}^{T^*_n(x)}\rhodue_n dz  =
\int_{-\infty}^{S(x)}\rhodue dz
$$
and then we get that the function $S$ is indeed the optimal map defined by \eqref{def:T*}.

\medskip

(ii.C) Let us now define the auxiliary positive function $\varphi_n:\R\to\R$ by
\begin{equation}\label{def:phi_n}
\varphi_n(x):=\sum_{i=1}^{n-1}|T_n^*(\yuno_i)-\yuno_i|^p \frac{\ell_n}{\yuno_{i+1}-\yuno_i}\chi_{[\yuno_i,\yuno_{i+1})}(x)+|T^*_n(\yuno_n)-\yuno_n|^p\chi_{\big[\yuno_n,\yuno_n+\ell_n\big)}(x).
\end{equation}
Note that $\varphi_n(x)=0$ if $x\in\R\backslash\big[\yuno_1,\yuno_n+\ell_n\big)$. 
We also have
$$
\left(\int_\R\varphi_n(x)dx\right)^{1/p}=
\left(\sum_{i=1}^n|T_n^*(\yuno_i)-\yuno_i|^p\ell_n\right)^{1/p}=
\left(\sum_{i=1}^n|\ydue_i-\yuno_i|^p\ell_n\right)^{1/p}=
\Dftl(\ybfuno,\ybfdue),
$$
where the second equality follows from the part (i) of the proof. 
Then it remains to prove that 
\begin{equation}\label{tesiequivalente}
\lim_{n\to +\infty}\int_\R\varphi_n(x)dx=\left(\Dlwr(\rhouno,\rhodue)\right)^p.
\end{equation}
Let us first show that there exists $\varphi:\R\to\R$ such that
\begin{equation}\label{def:phi}
\varphi=\lim_{n\to+\infty}\varphi_n, \qquad
\spt(\varphi)=\spt(\rhouno), \qquad
\varphi(x)=|T^*(x)-x|^p \rhouno(x) \quad \text{for a.e.\ } x\in\R.
\end{equation}
To this end, choose any $\bar x\in\R$ and $n>1$. We have four cases:

\medskip

\begin{tabular}{llll}
1. if & $\bar x<\yuno_1$ & then & $\varphi_n(\bar x)=0$; \\
2. if & $\yuno_1\leq\bar x<\yuno_n$ & then & $\varphi_n(\bar x)=|T_n^*(\yuno_{\bar k})-\yuno_{\bar k}|^p \frac{\ell_n}{\yuno_{\bar k+1}-\yuno_{\bar k}} \quad$ for some $\bar k=\bar k(\bar x)\in\{1,\ldots,n\}$ s.t.\ $\yuno_{\bar k}\leq\bar x<\yuno_{\bar k+1}$; \\
3. if & $\yuno_n\leq\bar x<\yuno_n+\ell_n$ & then &  $\varphi_n(\bar x)=|T_n^*(\yuno_n)-\yuno_n|^p$; \\
4. if & $\bar x\geq\yuno_n+\ell_n$  & then  & $\varphi_n(\bar x)=0$. 
\end{tabular}

\medskip

Noting that cases 1 and 4 are trivial, case 3 disappears for $n\to+\infty$, and, in case 2,  $T_n^*\to T^*$ and $\yuno_{\bar k}\to\bar x$ for $n\to+\infty$, we can apply same results as in Theorem \ref{teo:convergence_singleroad} (which makes it rigorous the fundamental limit relation $\frac{\ell_n}{\yuno_{\bar k+1}-\yuno_{\bar k}} \to\rhouno(\bar x)$) and we easily get \eqref{def:phi}.

Let us now prove that $\varphi_n\leq g$ for some $g\in\ L^1(\R)$ for all $n$. This comes from the fact that $\frac{\ell_n}{\yuno_{i+1}-\yuno_i}\leq 1$ (see Remark \ref{rem:nosmash}) and the fact that 
$$
|T_n^*(x)-x|\leq \sup_{t<t_f}\spt(\rhodue_n)-\inf_{t<t_f}\spt(\rhouno_n)
$$ 
(since the mass transfer only happens where there is some mass to transfer from/into) and, in turn, $\rhouno_n$, $\rhodue_n$ have compact support (see \eqref{rhon_a_supporto_compatto}).
Finally, the dominated convergence theorem allows us to write
$$
\lim_{n\to +\infty}\int_\R\varphi_n(x)dx=
\int_\R\varphi(x)dx=
\int_\R|T^*(x)-x|^p \rhouno(x)dx=
\left(\Dlwr\right)^p.
$$
\end{proof}

\subsection{Numerical results}\label{sec:singleroad_numericalresults}
In this section we aim at confirming the theoretical results presented above from the numerical point of view. In particular, we run micro- and macro-simulations \eqref{Cauchyproblems_micro}-\eqref{Cauchyproblems_macro} with either different initial conditions or velocity functions, and then we plot the function
\begin{equation}\label{n->Xi=|Dftl-Dlwr|}
n \mapsto \Xi_p(n):=\left|\Dftl(\ybfuno(t_f),\ybfdue(t_f);n)-\Dlwr(\rhouno(t_f,\cdot),\rhodue(t_f,\cdot)\right|,
\end{equation}
where we have explicitly stressed the dependence of $\Dftl$ on $n$ via the operator $E_n$.

\paragraph{Test 1} 
We consider a road segment $[0,100]$ and two shifted initial conditions
$$
\rhobaruno(x)=\frac12 \chi_{[5,20]}(x), \qquad 
\rhobardue(x)=\frac12 \chi_{[10,25]}(x),\qquad x\in[0,100].
$$
We set $\vmax=1$, $t_f=20$, and we employ Dirichlet zero boundary conditions. Micro-simulations are run starting from the initial conditions corresponding to the density functions defined above. 
In Fig.\ \ref{fig:sim:T1}
\begin{figure}[h!]
\centerline{
\includegraphics[width=0.47\textwidth]{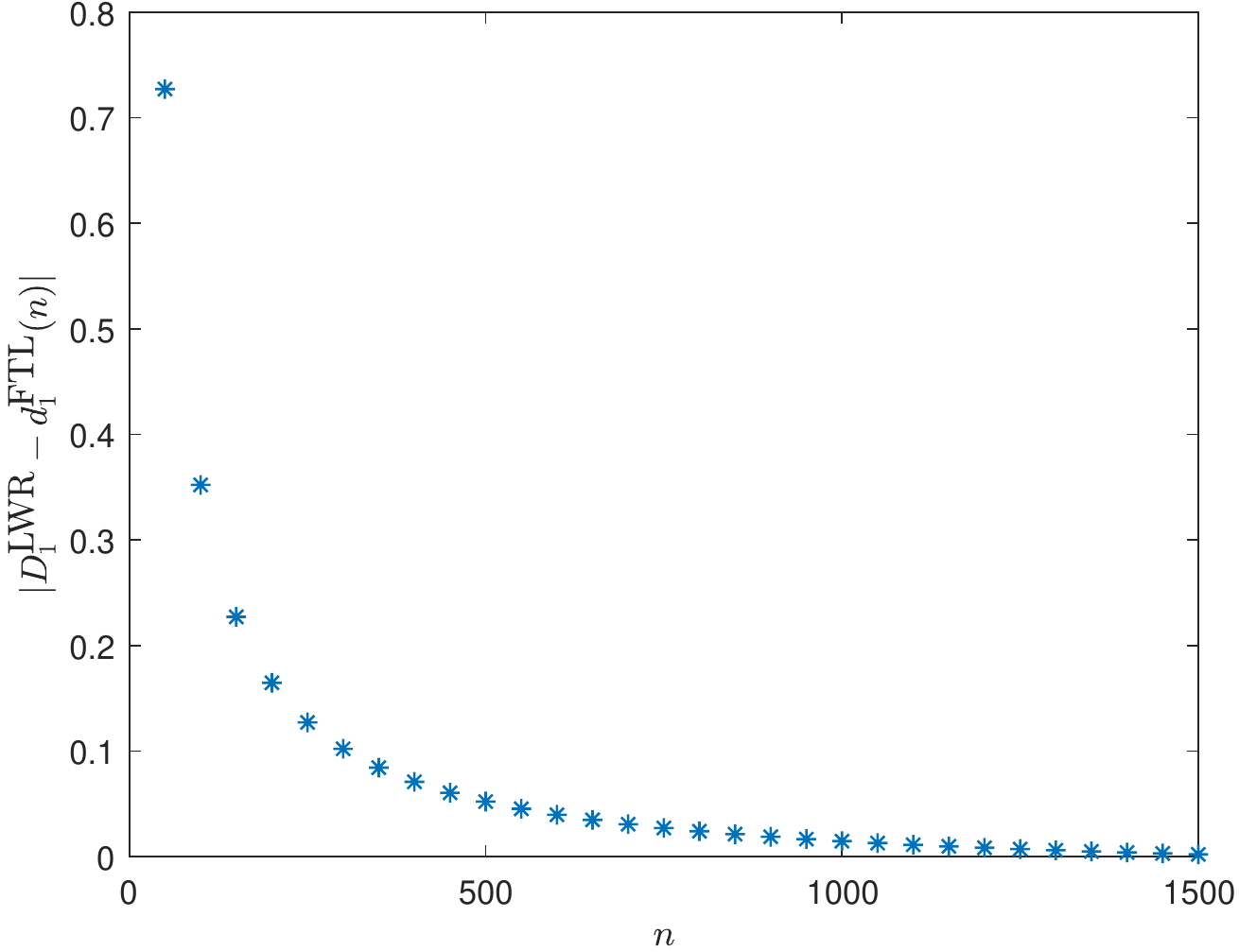}\qquad
\includegraphics[width=0.47\textwidth]{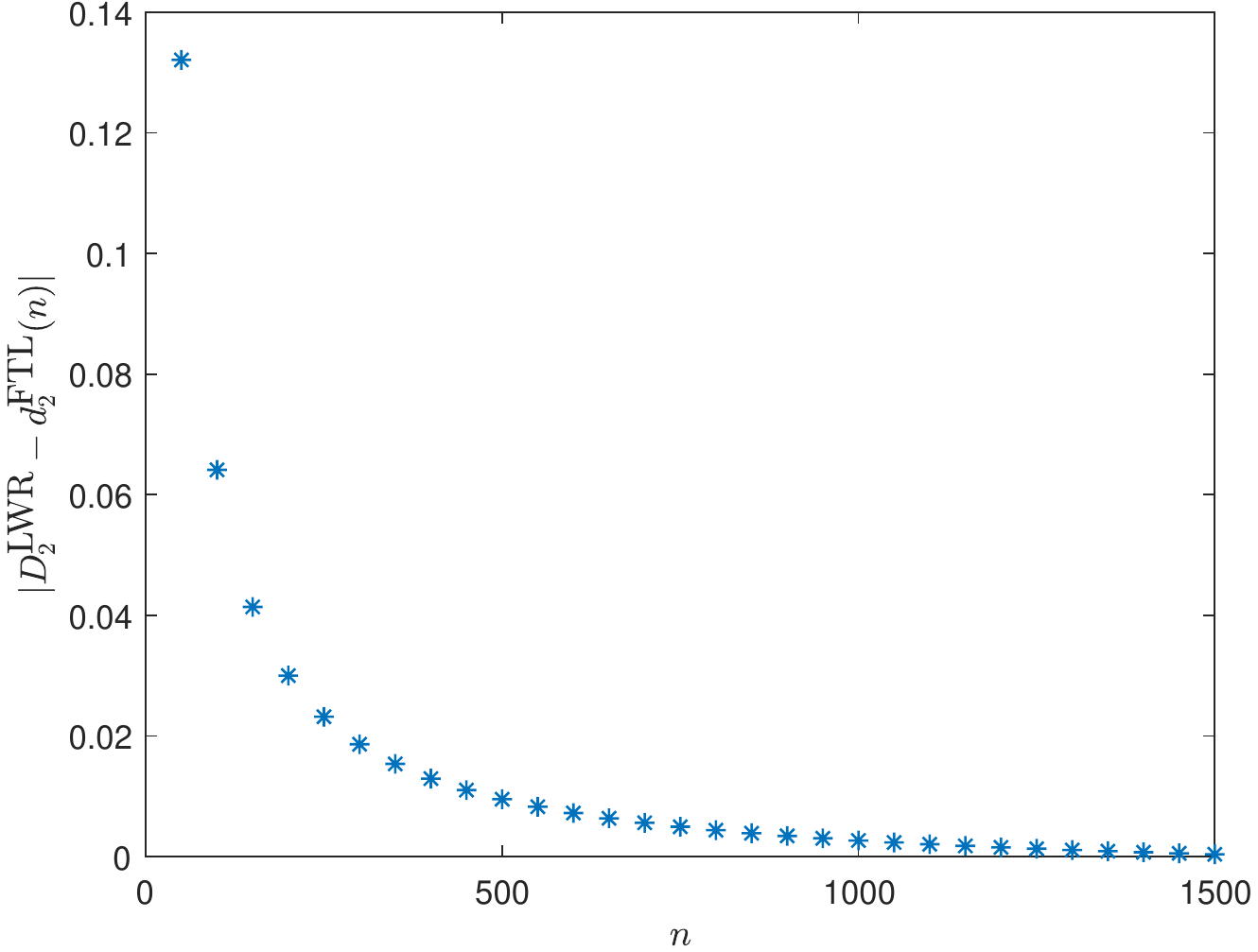}
}
\caption{Test 1: function $\Xi_1$ (left) and $\Xi_2$ (right). \label{fig:sim:T1}}
\end{figure}
we plot the functions $\Xi_1$ and $\Xi_2$. 
As it can be seen, the convergence to 0 is monotone and fast in both cases.

\paragraph{Test 2} 
In this test the two scenarios differ for the velocity functions rather than the initial conditions. 
We consider a road segment $[0,100]$ and the initial condition
$$
\bar\rho(x)=\frac12 \chi_{[10,25]}(x),\qquad x\in[0,100].
$$
We set $\vmaxuno=1$, $\vmaxdue=2$, $t_f=14$, and we employ Dirichlet zero boundary conditions. Micro-simulations are run starting from the initial condition corresponding to the density function defined above.
In Fig.\ \ref{fig:sim:T2}
\begin{figure}[h!]
\centerline{
\includegraphics[width=0.47\textwidth]{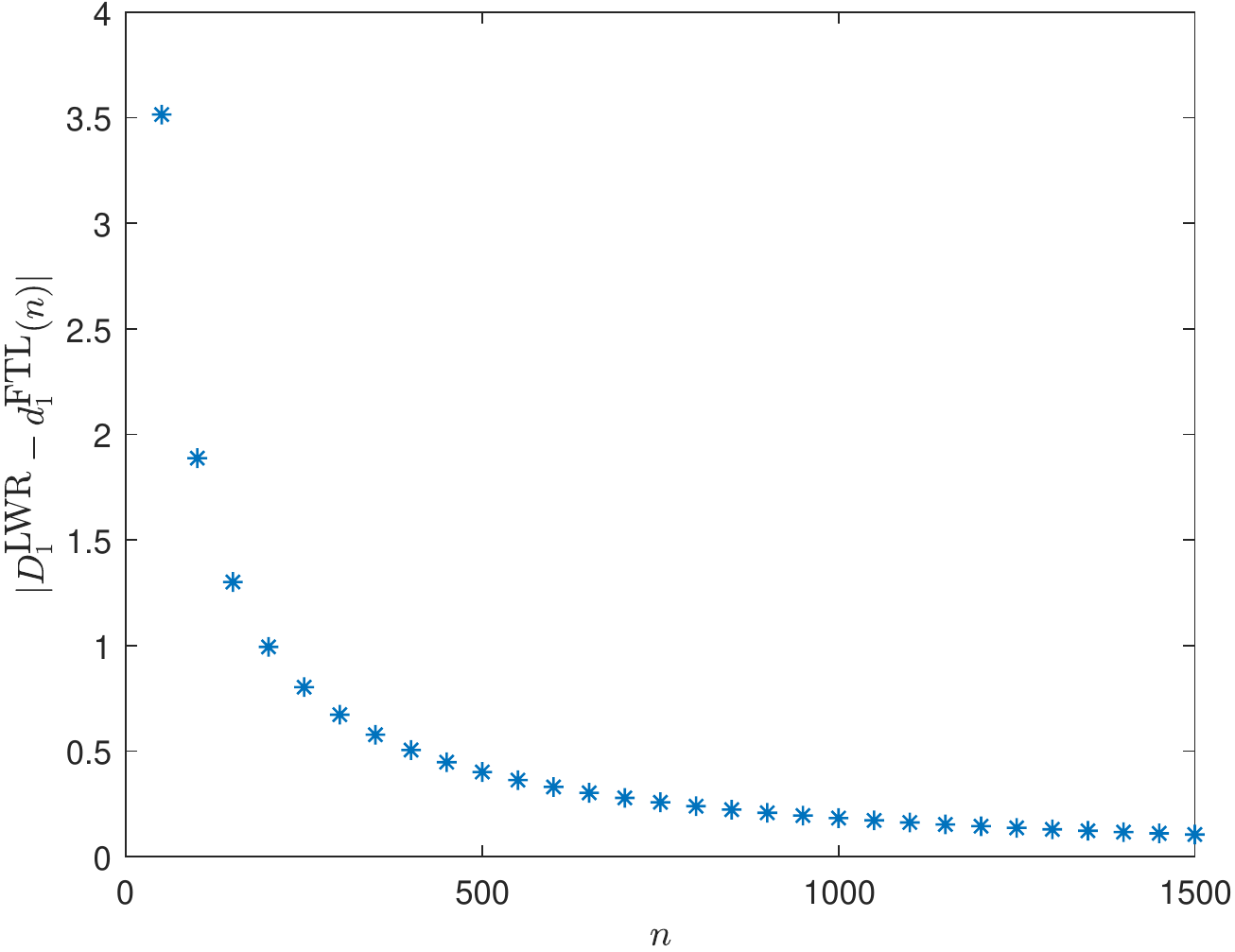}\qquad
\includegraphics[width=0.47\textwidth]{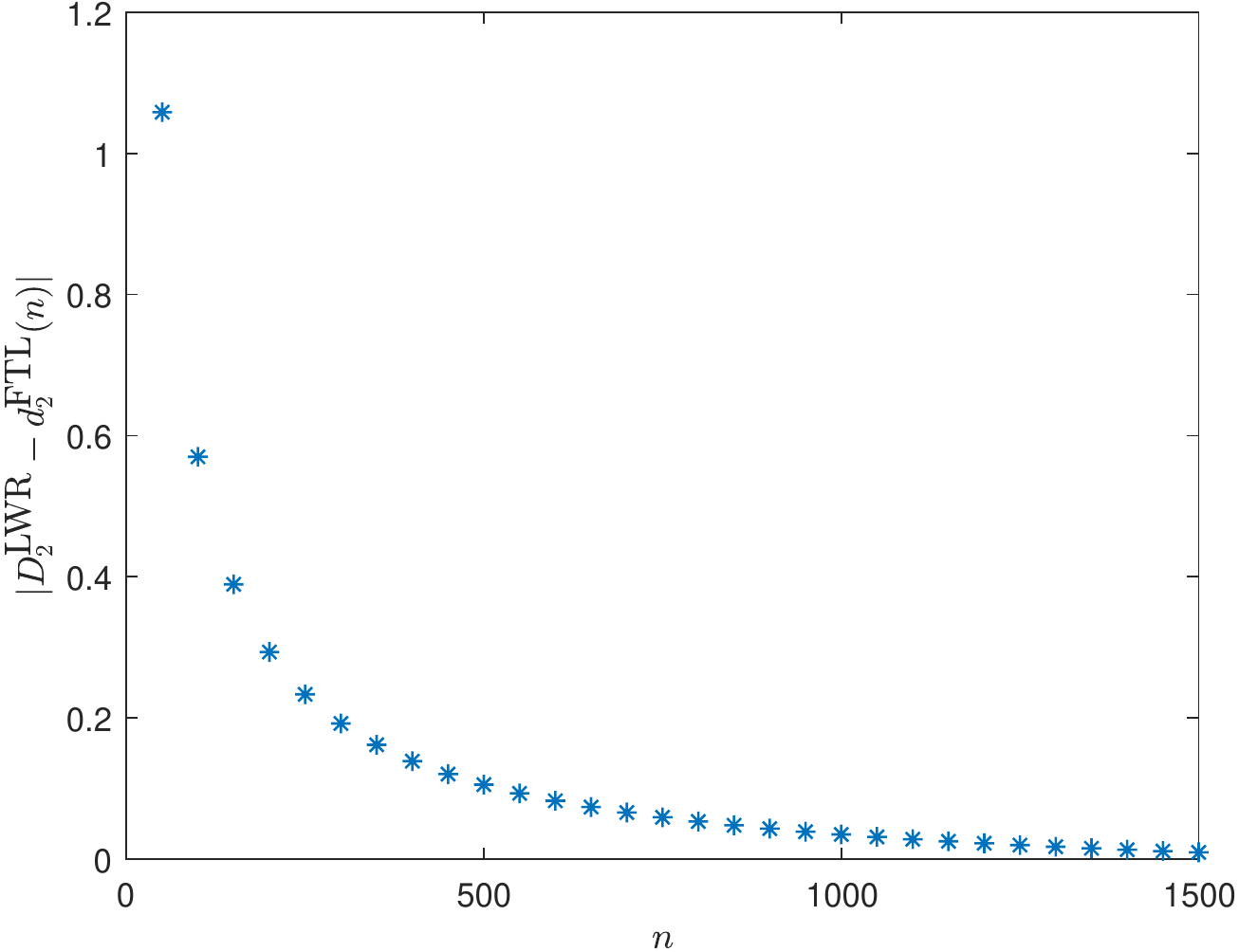}
}
\caption{Test 2: function $\Xi_1$ (left) and $\Xi_2$ (right). \label{fig:sim:T2}}
\end{figure}
we plot the functions $\Xi_1$ and $\Xi_2$. 
Again, the convergence to 0 is monotone and fast in both cases.

\section{Comparing comparisons on a road network}\label{sec:networkanalysis}

\subsection{Analytical results}\label{sec:network_analyticalresults}
The case of a road network is subtler and more difficult to categorize. The crucial difference is that, on some networks, vehicles loose their initial order (see Section \ref{sec:FtLnetwork}), and, in parallel, the Wasserstein distance does not enjoy the property of monotone rearrangement. Therefore, analytical results described in the previous section are not expected to hold, at least for a general road network.

\begin{example}\label{ex:1} (soft counterexample). 
Let us consider the case of a \emph{merge}, i.e.\ a simple network with one junction, two incoming roads and one outgoing road. Assume that vehicles are initially located on the two incoming roads only, as shown in Fig.\ \ref{fig:counterexample_soft_IC_scenario1}, and that, at final time, all vehicles reach the outgoing road, merging one after the other, as shown in Fig.\  \ref{fig:counterexample_soft_FC}(top).
Swapping the two initial conditions on the incoming roads, see Fig.\ \ref{fig:counterexample_soft_IC_scenario2}, we get the solution shown in Fig.\  \ref{fig:counterexample_soft_FC}(bottom). 
It is clear that, in this case, we have $\Dftlnop>0$ since vehicles do not occupy the same positions in the two scenarios. At the same time, we have $W=0$ since, \textit{on the whole}, vehicles do occupy the same positions.
\end{example}

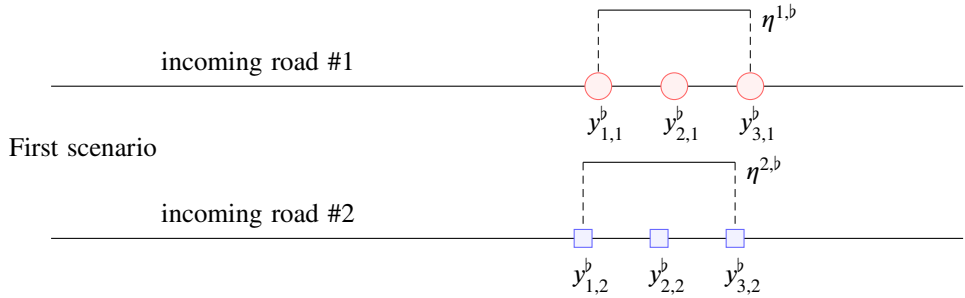
\begin{figure}[h!]
\centering
\begin{tikzpicture}
\draw (1,6.5) -- (13,6.5);

\node[text width=11mm] at (10.9,7.4) {$\eta^{1,\source}$};
\node[text width=11mm] at (10.7,5.4) {$\eta^{2,\source}$};
\draw [densely dashed,ultra thin](8.2,7.5) -- (8.2,6.5);
\draw [densely dashed,ultra thin](10.2,7.5) -- (10.2,6.5);
\draw (8.2,7.5) -- (10.2,7.5);

\draw [densely dashed,ultra thin](8,5.5) -- (8,4.5);
\draw [densely dashed,ultra thin](10,5.5) -- (10,4.5);
\draw (8,5.5) -- (10,5.5);
\node[text width=11mm] at (1,5.7) {\text{First scenario}};
\node[text width=11mm] at (3,6.8) {\text{incoming road \#1}};
\filldraw[color=red!60,fill=red!5,thin] (8.2,6.5) circle (5pt);
\filldraw[color=red!60,fill=red!5,thin] (9.2,6.5) circle (5pt);
\filldraw[color=red!60,fill=red!5,thin] (10.2,6.5) circle (5pt);
\node[text width=11mm] at (8.6,6.0) {$\yuno_{1,1}$};
\node[text width=11mm] at (9.6,6.0) {$\yuno_{2,1}$};
\node[text width=11mm] at (10.6,6.0) {$\yuno_{3,1}$};

\draw (1,4.5) -- (13,4.5);
\node[text width=11mm] at (3,4.8) {\text{incoming road \#2}};
\node[draw=blue!60,fill=blue!5,thin] at (8,4.5){};
\node[draw=blue!60,fill=blue!5,thin] at (9,4.5){};
\node[draw=blue!60,fill=blue!5,thin] at (10,4.5){};
\node[text width=11mm] at (8.4,4.0) {$\yuno_{1,2}$};
\node[text width=11mm] at (9.4,4.0) {$\yuno_{2,2}$};
\node[text width=11mm] at (10.4,4.0) {$\yuno_{3,2}$};
\end{tikzpicture}
\caption{Example \ref{ex:1}: Vehicles at initial time in the first scenario ($\source$).}
\label{fig:counterexample_soft_IC_scenario1}
\end{figure}


\begin{figure}[h!]
\centering
\begin{tikzpicture}
\draw (1,6.5) -- (13,6.5);

\node[text width=11mm] at (10.7,7.4) {$\eta^{1,\dest}$};
\node[text width=11mm] at (10.9,5.4) {$\eta^{2,\dest}$};
\draw [densely dashed,ultra thin](8,7.5) -- (8,6.5);
\draw [densely dashed,ultra thin](10,7.5) -- (10,6.5);
\draw (8,7.5) -- (10,7.5);

\draw [densely dashed,ultra thin](8.2,5.5) -- (8.2,4.5);
\draw [densely dashed,ultra thin](10.2,5.5) -- (10.2,4.5);
\draw (8.2,5.5) -- (10.2,5.5);
\node[text width=11mm] at (1,5.7) {\text{Second scenario}};
\node[text width=11mm] at (3,6.8) {\text{incoming road \#1}};
\filldraw[color=red!60,fill=red!5,thin] (8,6.5) circle (5pt);
\filldraw[color=red!60,fill=red!5,thin] (9,6.5) circle (5pt);
\filldraw[color=red!60,fill=red!5,thin] (10,6.5) circle (5pt);
\node[text width=11mm] at (8.4,6.0) {$\ydue_{1,1}$};
\node[text width=11mm] at (9.4,6.0) {$\ydue_{2,1}$};
\node[text width=11mm] at (10.4,6.0) {$\ydue_{3,1}$};

\draw (1,4.5) -- (13,4.5);
\node[text width=11mm] at (3,4.8) {\text{incoming road \#2}};
\node[draw=blue!60,fill=blue!5,thin] at (8.2,4.5){};
\node[draw=blue!60,fill=blue!5,thin] at (9.2,4.5){};
\node[draw=blue!60,fill=blue!5,thin] at (10.2,4.5){};
\node[text width=11mm] at (8.6,4.0) {$\ydue_{1,2}$};
\node[text width=11mm] at (9.6,4.0) {$\ydue_{2,2}$};
\node[text width=11mm] at (10.6,4.0) {$\ydue_{3,2}$};
\end{tikzpicture}
\caption{Example \ref{ex:1}: Vehicles at initial time in the second scenario ($\dest$).}
\label{fig:counterexample_soft_IC_scenario2}
\end{figure}
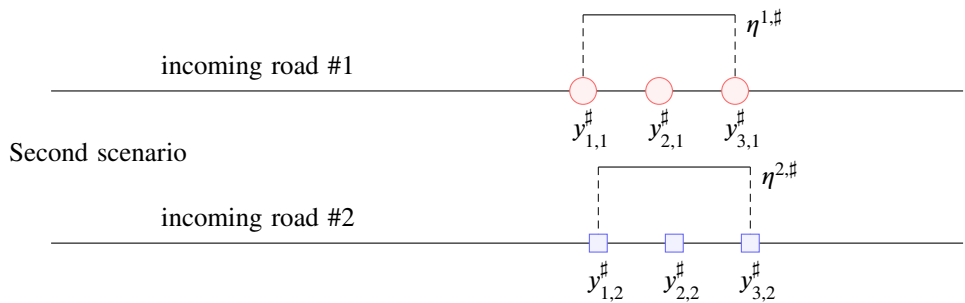


\begin{figure}[h!]
\centering
\begin{tikzpicture}
\draw (1,6.5) -- (13,6.5);
\node[text width=11mm] at (2,6.8) {\text{outgoing road (first scenario)}};
\filldraw[color=red!60,fill=red!5,thin] (12.6,6.5) circle (5pt);
\filldraw[color=red!60,fill=red!5,thin] (9.1,6.5) circle (5pt);
\filldraw[color=red!60,fill=red!5,thin] (7.4,6.5) circle (5pt);
\node[draw=blue!60,fill=blue!5,thin] at (6.8,6.5){};
\node[draw=blue!60,fill=blue!5,thin] at (8.1,6.5){};
\node[draw=blue!60,fill=blue!5,thin] at (10.6,6.5){};
\node[text width=11mm] at (13,6.0) {$\yuno_{3,1}$};
\node[text width=11mm] at (11,6.0) {$\yuno_{3,2}$};
\node[text width=11mm] at (9.5,6.0) {$\yuno_{2,1}$};
\node[text width=11mm] at (8.5,6.0) {$\yuno_{2,2}$};
\node[text width=11mm] at (7.8,6.0) {$\yuno_{1,1}$};
\node[text width=11mm] at (7.2,6.0) {$\yuno_{1,2}$};
\draw (1,4.5) -- (13,4.5);
\node[text width=11mm] at (2,4.8) {\text{outgoing road (second scenario)}};
\filldraw[color=red!60,fill=red!5,thin] (10.6,4.5) circle (5pt);
\filldraw[color=red!60,fill=red!5,thin] (8.1,4.5) circle (5pt);
\filldraw[color=red!60,fill=red!5,thin] (6.8,4.5) circle (5pt);
\node[draw=blue!60,fill=blue!5,thin] at (7.4,4.5){};
\node[draw=blue!60,fill=blue!5,thin] at (9.1,4.5){};
\node[draw=blue!60,fill=blue!5,thin] at (12.6,4.5){};
\node[text width=11mm] at (13,4.0) {$\ydue_{3,2}$};
\node[text width=11mm] at (11,4.0) {$\ydue_{3,1}$};
\node[text width=11mm] at (9.5,4.0) {$\ydue_{2,2}$};
\node[text width=11mm] at (8.5,4.0) {$\ydue_{2,1}$};
\node[text width=11mm] at (7.8,4.0) {$\ydue_{1,2}$};
\node[text width=11mm] at (7.2,4.0) {$\ydue_{1,1}$};
\draw[->] (12.6,5.7) -- (12.6,4.8);
\draw[->] (12.6,5.7) -- (10.6,4.8);
\node[text width=11mm] at (13.3,5.3) {$T^*$};
\end{tikzpicture}
\caption{Example \ref{ex:1}: Vehicles at final time, both scenarios.}
\label{fig:counterexample_soft_FC}
\end{figure}
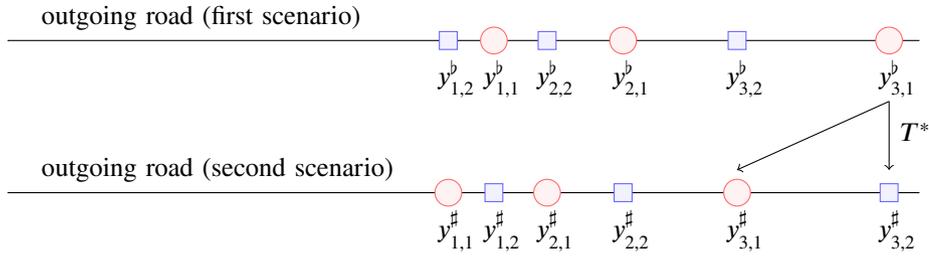

\begin{remark}
The case described above represents only a 'soft' counterexample to Theorem \ref{teo:w->W_singleroad}. 
In fact, we still get the convergence $\Dftlnop\to W$ for $n\to\infty$. This can be easily seen in Fig.\ \ref{fig:counterexample_soft_FC}, observing that, on a road with finite length filled with $n$ vehicles, the distance between one vehicle and the following one is $O(1/n)$. As a consequence, moving $n$ vehicles of mass $\ell_n$ for a distance $O(1/n)$ costs in total
$$
\Dftlnop=n \ell_n O(1/n)=O(1/n)\to 0.
$$
\end{remark}

\begin{example}\label{ex:2} (hard counterexample). 
Let us consider again the case described in Example \ref{ex:1}. Assume that in the first scenario vehicles are initially located as shown in Fig.\ \ref{fig:counterexample_soft_IC_scenario1}, while, in the second scenario, vehicles on the first incoming road are moved to the left so much to be, at final time, all behind the last vehicle coming from the second incoming road, see Fig.\ \ref{fig:counterexample_hard_outgoingroad}.

In this case the difference between $\Dftlnop$ and $W$ is kept also in the limit for $n\to\infty$ since the distance $|\yuno_{3,1}-\ydue_{3,1}|$ does not shrink to 0 in the limit (it must be larger than $\spt(\eta^{2,\sharp})$).
\end{example}

\begin{figure}[h!]
\centering
\begin{tikzpicture}
\draw (1,2) -- (13,2);
\node[text width=11mm] at (12,2.3) {\text{outgoing road (first scenario)}};
\filldraw[color=red!60,fill=red!5,thin] (2.9,2) circle (5pt);
\node[draw=blue!60,fill=blue!5,thin] at (2,2){};
\filldraw[color=red!60,fill=red!5,thin] (6,2) circle (5pt);
\node[draw=blue!60,fill=blue!5,thin] at (4.2,2){};
\filldraw[color=red,fill=red!5,thin] (9.5,2) circle (5pt);
\node[draw=blue!60,fill=blue!5,thin] at (7.5,2){};
\node[text width=11mm] at (2.4,1.6) {$\yuno_{1,2}$};
\node[text width=11mm] at (3.3,1.6) {$\yuno_{1,1}$};
\node[text width=11mm] at (4.6,1.6) {$\yuno_{2,2}$};
\node[text width=11mm] at (6.4,1.6) {$\yuno_{2,1}$};
\node[text width=11mm] at (7.9,1.6) {$\yuno_{3,2}$};
\node[text width=11mm] at (9.9,1.6) {$\yuno_{3,1}$};
\draw (1,0) -- (13,0);
\node[text width=11mm] at (12,0.3) {\text{outgoing road (second scenario)}};
\filldraw[color=red!60,fill=red!5,thin] (2.3,0) circle (5pt);
\filldraw[color=red!60,fill=red!5,thin] (3,0) circle (5pt);
\filldraw[color=red,fill=red!5,thin] (4,0) circle (5pt);
\node[draw=blue!60,fill=blue!5,thin] at (5.8,0) {};
\node[draw=blue!60,fill=blue!5,thin] at (8,0) {};
\node[draw=blue!60,fill=blue!5,thin] at (9.8,0){};
\node[text width=11mm] at (2.7,-0.4) {$\ydue_{1,1}$};
\node[text width=11mm] at (3.4,-0.4) {$\ydue_{2,1}$};
\node[text width=11mm] at (4.4,-0.4) {$\ydue_{3,1}$};
\node[text width=11mm] at (6.2,-0.4) {$\ydue_{1,2}$};
\node[text width=11mm] at (8.4,-0.4) {$\ydue_{2,2}$};
\node[text width=11mm] at (10.2,-0.4) {$\ydue_{3,2}$};
\draw[->] (9.5,1.3) -- (4.1,0.3);
\draw[->] (9.5,1.3) -- (9.8,0.3);
\node[text width=11mm] at (10.3,0.8) {$T^*$};
\end{tikzpicture}
\caption{Example \ref{ex:2}: Vehicles at final time, both scenarios.}
\label{fig:counterexample_hard_outgoingroad}
\end{figure}
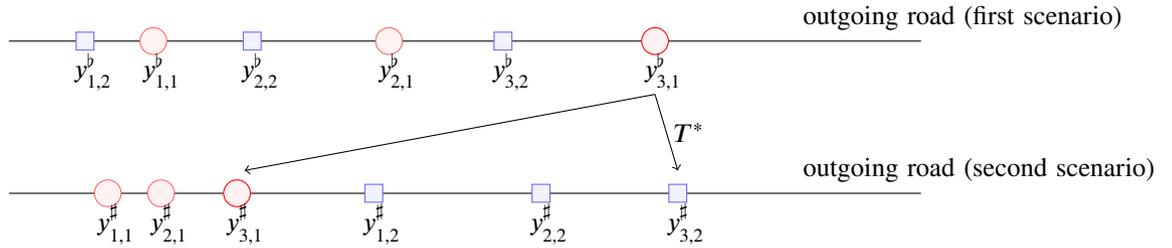


\subsection{Numerical results}\label{sec:network_numericalresults}
In this section we aim at confirming the theoretical results presented above from the numerical point of view. As we did in Section \ref{sec:singleroad_numericalresults}, we run micro- and macro-simulations with different initial conditions and then we plot the function $\Xi$.

\paragraph{Test 3 (soft counterexample)}
We consider the network of Example \ref{ex:1} where each road has length 20. On such a network, path 1 (resp., 2) joins the first (resp., second) incoming road with the outgoing road. The two slightly shifted initial conditions are
$$
\left\{
\begin{array}{ll}
\eta^{1,\source}(x)=\chi_{[0,5]+\frac{\ell_n}{2}}(x), \\ [2mm]
\eta^{2,\source}(x)=\chi_{[0,5]}(x),
\end{array}
\right. ,
\qquad
\left\{
\begin{array}{ll}
\eta^{1,\dest}(x)=\chi_{[0,5]}(x), \\ [2mm]
\eta^{2,\dest}(x)=\chi_{[0,5]+\frac{\ell_n}{2}}(x),
\end{array}
\right. ,
\qquad 
x\in[0,20].
\qquad
$$
Note that vehicles are bumper-to-bumper. We also set $\vmax=1$, $t_f=50$, and we employ Dirichlet zero boundary conditions. Micro-simulations are run starting from the initial conditions corresponding to the density functions defined above. 
In Fig.\ \ref{fig:sim:T3} 
\begin{figure}[h!]
\centerline{\includegraphics[width=0.47\textwidth]{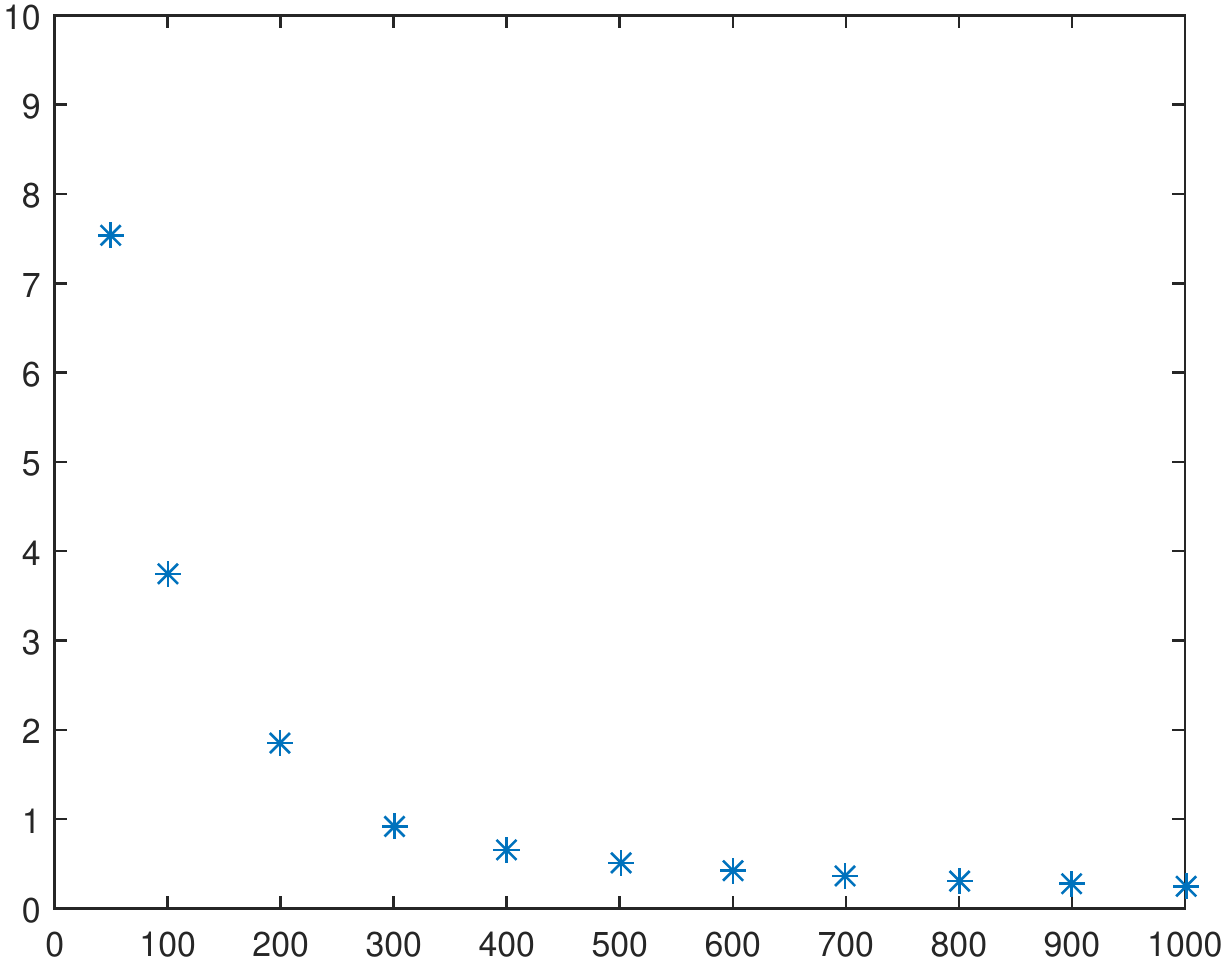}}
\caption{Test 3: function $\Xi_1$.
\label{fig:sim:T3}}
\end{figure}
we plot the function $\Xi_1$ at final time $t=t_f$. As expected, the convergence is achieved also in this case.

\paragraph{Test 4 (hard counterexample)} 
We consider again the network of Example \ref{ex:1} where each road has length 30. The two shifted initial conditions are
$$
\left\{
\begin{array}{ll}
\eta^{1,\source}(x)=\chi_{[20,25]}(x), \\ [2mm]
\eta^{2,\source}(x)=\chi_{[0,5]}(x),
\end{array}
\right. ,
\qquad
\left\{
\begin{array}{ll}
\eta^{1,\dest}(x)=\chi_{[0,5]}(x), \\ [2mm]
\eta^{2,\dest}(x)=\chi_{[20,25]}(x),
\end{array}
\right. ,
\qquad 
x\in[0,30].
\qquad
$$
We also set $\vmax=1$, $t_f=55$, and we employ Dirichlet zero boundary conditions. Micro-simulations are run starting from the initial conditions corresponding to the density functions defined above. 
In Fig.\ \ref{fig:sim:T4} 
\begin{figure}[h!]
\centerline{\includegraphics[width=0.47\textwidth]{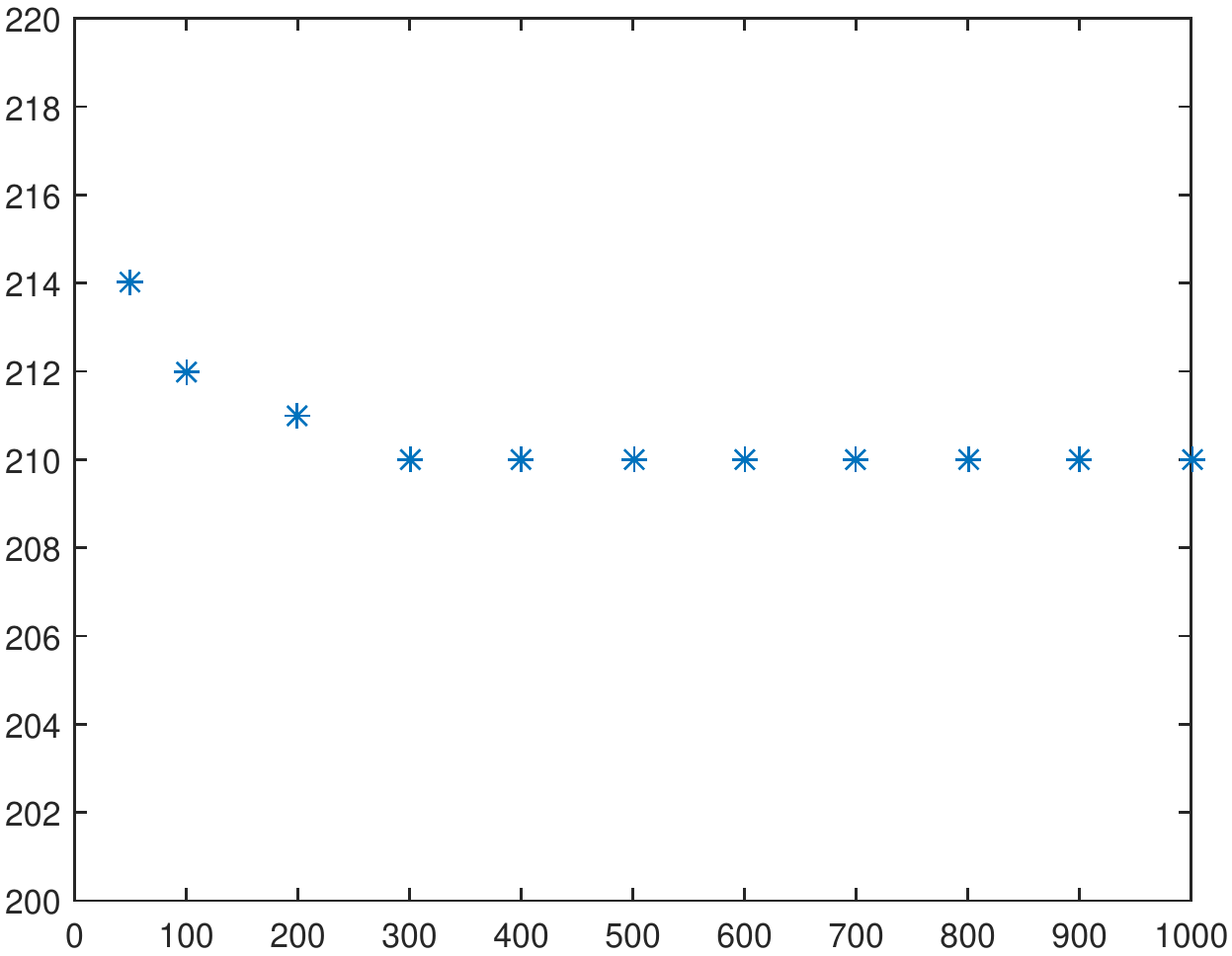}}
\caption{Test 4: function $\Xi_1$.
\label{fig:sim:T4}}
\end{figure}
we plot the function $\Xi_1$ at final time $t=t_f$. As expected, the convergence is not achieved in this case.

\section{Conclusions}\label{sec:conclusions}
In this paper we have investigated the problem of comparing traffic states in a multiscale framework. 
The notion of Wasserstein distance appears totally adequate in the case of single road simulations since it matches human intuition and it is totally independent of the scale of observation. 
In the case of a road network, instead, the microscopic Wasserstein distance is not the right notion of distance unless we stop distinguishing single vehicles, which is exactly what happens in the many-particle limit.



\section*{Acknowledgments}
Authors want to thank Benedetto Piccoli for the useful discussions about the properties of the Wasserstein distance, and Andrea Tosin for the useful suggestions. Authors also thank the anonymous referees for their suggestions on the proof of Theorem \ref{teo:w->W_singleroad}. \\
\noindent E. Cristiani is a member of the INdAM Research group GNCS.

\subsection*{Author contributions}
Both authors contributed equally to this work.

\bibliography{biblio_traffico}



\end{document}